\documentclass[11pt,twoside]{amsart}
\usepackage{graphicx}
\usepackage{color}
\usepackage{graphicx}
\usepackage{amssymb}
\usepackage{amsmath}
\usepackage{amsthm}
\usepackage{amsfonts}
\usepackage{amsmath, amsfonts, amssymb}
\usepackage{color}
\usepackage{multirow}
\usepackage{rotating}
\newtheorem{thm}{Theorem}[section]
\newtheorem{cor}[thm]{Corollary}

\newtheorem{rem}[thm]{\bf{Remark}}

\numberwithin{equation}{section}

\definecolor{Gray}{gray}{0.85}

\begin{document}
		\small \begin{center}{\textit{ In the name of
					Allah, the Beneficent, the Merciful}}\end{center}
		\vspace{0.5cm}
	\title{On identities of $2$-dimensional algebras}

	\author{H. Ahmed$^1$, U.Bekbaev$^2$, I.Rakhimov$^3$}
	
	\thanks{{\scriptsize
			emails: $^1$houida\_m7@yahoo.com; $^2$uralbekbaev@gmail.com; $^3$isamiddin@uitm.edu.my}}
	\maketitle
	\begin{center}
		{\scriptsize \address{$^1$Depart. of Math., Faculty of Science, Taiz University, Taiz, Yemen}} \\
{\scriptsize \address{$^2$Depart. of Mathematical and Natural Sciences, TTPU, Tashkent, Uzbekistan}}\\
{\scriptsize \address{$^3$Depart. of Math., Faculty of Computer and Mathematical Sciences, UiTM, Shah Alam, Malaysia \\ $\&$ V.I.Romanovski Institute of Mathematics, Uzbekistan Academy of Sciences}}
	\end{center}

%
%

	\begin{abstract} In the paper we provide some polynomial identities for finite-dimensional algebras. A list of well known single polynomial identities is exposed and the classification of all $2$-dimensional algebras with respect to these identities is given.
	\end{abstract}	
	\vskip 0.2cm
	
	\section{Introduction}
	
	It is known that many important algebras in use are so called PI-algebras, that is, algebras satisfying a certain set of polynomial identities. Therefore, the classification of such algebras, up to isomorphism, is of a great interest. Earlier we have given classification results for some important classes of two-dimensional PI-algebras \cite{B1,B3,B4,B5}. In this paper we consider a list of some important polynomial identities which have appeared earlier in the theory of algebras and  present a classification of two-dimensional algebras with respect to these identities. For other results related to the classification problem and the problems raised in this paper we refer the reader to \cite{D2003,G2011,P2000,Casado2,C,IK}.
	
	The organization of the paper is as follows. In the next section we introduce definitions, notations and results needed in the course of the study followed by two section where we present main results of the paper. In Section 3 we provide some polynomial identities for finite-dimensional algebras. The last section is devoted to the classification of two-dimensional algebras with respect to the identities specified.
	\vskip 0.4 true cm

	\section{Preliminaries}
 In this paper an algebra $(\mathbb{A},\cdot)$ means a vector space $\mathbb{A}$ over a field $\mathbb{F}$ with a given
	bilinear map \[\cdot: \mathbb{A}\times\mathbb{A}\rightarrow\mathbb{A},\ (\mathbf{u},\mathbf{v})\mapsto \mathbf{u}\cdot \mathbf{v}\] and often we drop $\cdot$ in the writing.
	
	If $\mathbb{A}$ such a two-dimensional algebra and $\mathrm{e}=(e_1,e_2)$ is a fixed basis by\\
	$A=\left(
	\begin{array}{cccc}
	\alpha_1 & \alpha_2 &\alpha_3 & \alpha_4 \\
	\beta_1 & \beta _2 & \beta_3 & \beta_4
	\end{array}\right)$ we denote its matrix of structural constants (MSC) with respect to this basis, i.e., \[e_1e_1=\alpha_1e_1+\beta_1e_2,\ e_1e_2=\alpha_2e_1+\beta_2e_2,\ e_2e_1=\alpha_3e_1+\beta_3e_2,\ e_2e_2=\alpha_4e_1+\beta_4e_2.\] Further it is assumed that the basis $\mathrm{e}$ is fixed and we don't make difference between an algebra $\mathbb{A}$ and its MSC A with respect to this basis.
	
	The classification problem of all two dimensional algebras over any field $\mathbb{F}$, over that any the second and third degree polynomial has a root, has been solved in \cite{B2}. The classification there was done via providing the canonical MSCs for such algebras. In this paper we rely on the result of \cite{B2}, follow its notations and for a convenience we present here the corresponding canonical representatives according to $\mathrm{Char}(\mathbb{F})\neq 2,3$, $\mathrm{Char}(\mathbb{F})= 2$ and $\mathrm{Char}(\mathbb{F})= 3$ cases. Note that the parameters given in the canonical representatives may take any values in $\mathbb{F}$.
\begin{center}	
\begin{tabular}{|c|c|c|c|c|c|c|c|c|c|}
  \hline
  \multirow{14}{*}{\begin{turn}{90}$\mathrm{Char}\left(\mathbb{F}\right)\neq 2,3$ \end{turn}} &\multirow{2}{*}{Algebra}&\multicolumn{8}{c|}{Structure constants}\\
  \cline{3-10}
&&$\alpha_1$&$\alpha_2$&$\alpha_3$&$\alpha_4$&$\beta_1$&$\beta_2$&$\beta_3$&$\beta_4$\\
  \cline{2-10}
  &$A_{1}(\mathbf{c}) $& $\alpha_1$ & $\alpha_2$ &$\alpha_2+1$ & $\alpha_4$ & $\beta_1$ &$-\alpha_1$ &$ -\alpha_1+1$ &$ -\alpha_2$\\
  \cline{2-10}
  &$A_{2}(\mathbf{c})$ &$ \alpha_1$ & $0$ & $0$ & $1$& $\beta _1$& $\beta _2$& $1-\alpha_1$&0 \\
  \cline{2-10}
  &$A_{3}(\mathbf{c})$ & $0$ & $1$ & $1$ & $0$ & $\beta _1$& $\beta _2$ & 1&$-1$ \\
  \cline{2-10}
 &$ A_{4}(\mathbf{c}) $& $\alpha _1$ & $0 $&$ 0$ & $0$ & $0$ & $\beta _2$& $1-\alpha _1$&$0$ \\
  \cline{2-10}
  &$A_{5}(\mathbf{c})$ & $\alpha_1$& $0$ & $0$ & $0$ &$1$ & $2\alpha_1-1$ & $1-\alpha_1$&$0$ \\
  \cline{2-10}
  &$A_{6}(\mathbf{c})$ & $\alpha_1$ & $0$ & $0$ &$ 1$&$\beta _1$& $1-\alpha_1$ & $-\alpha_1$&$0$\\
  \cline{2-10}
  &$ A_{7}(\mathbf{c}) $&$ 0$ & $1$ & $1$ & $0$&	$\beta_1$& $1$& $0$&$-1$ \\
  \cline{2-10}
   &$A_{8}(\mathbf{c})$ & $\alpha_1 $& $0$ &$ 0$ & $0$& $0$ & $1-\alpha_1$ & $-\alpha_1$&$0$ \\
  \cline{2-10}
  &$A_{9}$ & $\frac{1}{3}$& $0$ & $0$ & $0$ &$1$ & $\frac{2}{3}$ & $-\frac{1}{3}$&$0$ \\
  \cline{2-10}
&$A_{10}$ &$0$ &$1$ &$1$ & $0$&	$0$ &$0$& $0$ &$-1 $\\
  \cline{2-10}
 & $A_{11} $& $0$ &$1$ &$1$ & $0$&$1$ &$0$& $0$ &$-1 $\\
  \cline{2-10}
 &$ A_{12}$ & $0$ & $0$ & $0$ & $0$&$1$ &$0$&$0$ &$0$ \\
  \hline
\end{tabular}
\end{center}
$A_{2}(\mathbf{c})=\left(
	\begin{array}{cccc}
	\alpha_1 & 0 & 0 & 1 \\
	\beta _1& \beta _2& 1-\alpha_1&0
	\end{array}\right)\cong \left(
	\begin{array}{cccc}
	\alpha_1 & 0 & 0 & 1 \\
	-\beta _1& \beta _2& 1-\alpha_1&0
	\end{array}\right)$, {where $\mathbf{c}=(\alpha_1, \beta_1, \beta_2)\in \mathbb{F}^3,$ \\ $A_{6}(\mathbf{c})=\left(
	\begin{array}{cccc}
	\alpha_1 & 0 & 0 & 1 \\
	\beta _1& 1-\alpha_1 & -\alpha_1&0
	\end{array}\right)\cong \left(
	\begin{array}{cccc}
	\alpha_1 & 0 & 0 & 1 \\
	-\beta _1& 1-\alpha_1 & -\alpha_1&0
	\end{array}\right),$ where $\mathbf{c}=(\alpha_1, \beta_1)\in \mathbb{F}^2,$

\begin{center}
\begin{tabular}{|c|c|c|c|c|c|c|c|c|c|}
  \hline
\multirow{14}{*}{\begin{turn}{90}$\mathrm{Char}\left(\mathbb{F}\right)=2$ \end{turn}} &\multirow{2}{*}{Algebra}&\multicolumn{8}{c|}{The structure constants}\\
  \cline{3-10}
&&$\alpha_1$&$\alpha_2$&$\alpha_3$&$\alpha_4$&$\beta_1$&$\beta_2$&$\beta_3$&$\beta_4$\\
  \cline{2-10}
  &$A_{1,2}(\mathbf{c}) $& $\alpha_1$ & $\alpha_2$ &$1+\alpha_2$ & $\alpha_4$ & $\beta_1$ &$\alpha_1$ &$ 1+\alpha_1$ &$ \alpha_2$\\
  \cline{2-10}
  &$A_{2,2}(\mathbf{c})$ &$\alpha_1 $& $0$ & $0$ &$1$ &$\beta _1$& $\beta_2 $& $1+\alpha_1$&$0$ \\
  \cline{2-10}
  &$A_{3,2}(\mathbf{c})$ & $\alpha_1 $&$1$ &$1$ & $0$&$0$& $\beta_2$ & $1+\alpha_1$&$1$ \\
  \cline{2-10}
 &$ A_{4,2}(\mathbf{c}) $& $\alpha _1 $& $0$ & $0$ & $0$&$0$ & $\beta_2$ & $1+\alpha _1$&$0$ \\
  \cline{2-10}
  &$A_{5,2}(\mathbf{c})$ & $\alpha_1 $& $0$ & $0$ & $0$ &$1$ &$1$ & $1+\alpha_1$&$0$ \\
  \cline{2-10}
  &$A_{6,2}(\mathbf{c})$ & $\alpha_1$ & $0$ & $0$ &$1$&$ \beta _1$& $1+\alpha_1 $&$ \alpha_1$&$0$\\
  \cline{2-10}
  &$ A_{7,2}(\mathbf{c}) $&$\alpha_1$ &$1$ &$1$ & $0$ &$0$& $1+\alpha_1$& $\alpha_1$&$1$ \\
  \cline{2-10}
   &$A_{8,2}(\mathbf{c})$ &$ \alpha_1 $& $0$ & $0$ & $0$&$0$ & $1+\alpha_1$ &$ \alpha_1$&$0$ \\
  \cline{2-10}
  &$A_{9,2}$ &$1$ & $0$ & $0$ & $0$&$1$ & $0$ &$1$&$0$\\
  \cline{2-10}
&$A_{10,2}$ &$0$ &$1$ &$1$ & $0$&$0$ &$0$& $0$ &$1$ \\
  \cline{2-10}
 & $A_{11,2} $&$1$ &$1$ &$1$ & $0$&$0$ &$1$&$1$ &$1$\\
  \cline{2-10}
 &$ A_{12,2}$ & $0$ & $0$ & $0$ & $0$&$1$ &$0$&$0$ &$0$ \\
  \hline
\end{tabular}
\end{center}
\begin{center}
\begin{tabular}{|c|c|c|c|c|c|c|c|c|c|}
  \hline
\multirow{14}{*}{\begin{turn}{90}$\mathrm{Char}\left(\mathbb{F}\right)=3$ \end{turn}} &\multirow{2}{*}{Algebra}&\multicolumn{8}{c|}{The structure constants}\\
  \cline{3-10}
&&$\alpha_1$&$\alpha_2$&$\alpha_3$&$\alpha_4$&$\beta_1$&$\beta_2$&$\beta_3$&$\beta_4$\\
  \cline{2-10}
  &$A_{1,3}(\mathbf{c}) $& $\alpha_1$ & $\alpha_2$ &$\alpha_2+1$ & $\alpha_4$ & $\beta_1$ &$-\alpha_1$ &$ -\alpha_1+1$ &$ -\alpha_2$\\
  \cline{2-10}
  &$A_{2,3}(\mathbf{c})$ &$ \alpha_1$ & $0$ & $0$ &$1$& $\beta _1$& $\beta _2$& $1-\alpha_1$&$0$ \\
  \cline{2-10}
  &$A_{3,3}(\mathbf{c})$ & $0$ &$1$ &$1$ & $0$ & $\beta _1$& $\beta _2$ &$1$&$-1$ \\
  \cline{2-10}
 &$ A_{4,3}(\mathbf{c}) $& $\alpha _1$ & $0$ & $0$ & $0$ & $0$ & $\beta _2$& $1-\alpha _1$&$0$ \\
  \cline{2-10}
  &$A_{5,3}(\mathbf{c})$ & $\alpha_1$& $0$ & $0$ & $0$ &$1$ & $-\alpha_1-1$ & $1-\alpha_1$&$0$ \\
  \cline{2-10}
   &$A_{6,3}(\mathbf{c})$ & $\alpha_1$ & $0$ & $0$ &$1$&$\beta _1$& $1-\alpha_1$ & $-\alpha_1$&$0$\\
  \cline{2-10}
  &$ A_{7,3}(\mathbf{c}) $& $0$ &$1$ &$1$ & $0$&	$\beta_1$&$1$& $0$&$-1$ \\
  \cline{2-10}
   &$A_{8,3}(\mathbf{c})$ & $\alpha_1 $& $0$ & $0$ & $0$& $0$ & $1-\alpha_1$ & $-\alpha_1$&$0$ \\
  \cline{2-10}
  &$A_{9,3}$ & $0$ &$1$&$1$& $0$&$1$ &$0$&$0$ &$-1$\\
  \cline{2-10}
&$A_{10,3}$ &$0$ &$1$ &$1$ & $0$&$0$ &$0$&$0$ &$-1$ \\
  \cline{2-10}
 & $A_{11,3} $&$1$ & $0$ & $0$ & $0$&$1$ &$-1$&$-1$ &$0$\\
  \cline{2-10}
 &$ A_{12,3}$ & $0$ & $0$ & $0$ & $0$&$1$ &$0$&$0$ &$0$ \\
  \hline
\end{tabular}
\end{center}
$A_{2,3}(\mathbf{c})=\left(
\begin{array}{cccc}
\alpha_1 & 0 & 0 & 1 \\
\beta _1& \beta _2& 1-\alpha_1&0
\end{array}\right)\cong\left(
\begin{array}{cccc}
\alpha_1 & 0 & 0 & 1 \\
-\beta _1& \beta _2& 1-\alpha_1&0
\end{array}\right),\ \mbox{where}\ \mathbf{c}=(\alpha_1, \beta_1, \beta_2)\in \mathbb{F}^3,$ \\ $A_{6,3}(\mathbf{c})=\left(
\begin{array}{cccc}
\alpha_1 & 0 & 0 & 1 \\
\beta _1& 1-\alpha_1 & -\alpha_1&0
\end{array}\right)\cong \left(
\begin{array}{cccc}
\alpha_1 & 0 & 0 & 1 \\
-\beta _1& 1-\alpha_1 & -\alpha_1&0
\end{array}\right),\ \mbox{where}\ \mathbf{c}=(\alpha_1, \beta_1)\in \mathbb{F}^2.$
		\section{Polynomial identities of two-dimensional algebras} Let $\mathbb{F}$ be a field, $m\leq n\leq l$ be  fixed natural numbers, $S_n$ be the symmetric group and $x^1,x^2,...,x^l$ be non-commutative, non-associative variables. By the use of these variables and parenthesis $(,)$ one can construct different non-associative monomials(words) containing each of the variables  $x^1, x^2,..., x^l$ only once and they may occur in the monomial in any order. For example, in the case of $l=2$ one has only two such monomials $x^1x^2$ and $x^2x^1$, whereas in the case of $l=3$ there are the following twelve possibilities  \[\{(x^{\sigma(1)}x^{\sigma(2)})x^{\sigma(3)}: \sigma\in S_3\}\cup \{x^{\sigma(1)}(x^{\sigma(2)}x^{\sigma(3)}): \sigma\in S_3\}.\]  Further	
	 $w(x^1,x^2,...,x^l)$ stands for such a monomial and we associate with $w(x^1,x^2,...,x^l)$  the following multi-linear polynomial \[w_{n,d}(x^1,x^2,...,x^l)= \sum_{\sigma\in S_n}sgn(\sigma)w(x^{\sigma(1)},x^{\sigma(2)},...,x^{\sigma(n)},x^{n+1}, x^{n+2},..., x^{l}).\]
	 	
	Let $\mathbb{A}$ be any $m$-dimensional algebra over $\mathbb{F}$. We use the notations $[\mathbf{u},\mathbf{v}]=\mathbf{u}\cdot \mathbf{v}-\mathbf{v}\cdot \mathbf{u}$ and  $[\mathbf{u},\mathbf{v},\mathbf{w}]=(\mathbf{u}\cdot \mathbf{v})\cdot \mathbf{w}-\mathbf{u}\cdot (\mathbf{v}\cdot \mathbf{w})$ for the commutator and the associator of $\mathbf{u},\mathbf{v}$ and $\mathbf{u},\mathbf{v},\mathbf{w}$, respectively. It is assumed that a basis  $\mathrm{e}=\{e_j\}_{j=1,2,3,...,m}$ of $\mathbb{A}$ over $\mathbb{F}$ is fixed and $\mathbf{u}^i=\sum\limits_{j=1}^{m}x^i_je_j\in \mathbb{A}$, where $i=1,2,,3,...,l$, stand for any elements of $\mathbb{A}$. Define \[|\mathbf{u}^1,\mathbf{u}^2,\dots ,\mathbf{u}^m|=\left|\begin{array}{cccc}x^1_1&x^1_2& \dots &x^1_m\\ x^2_1&x^2_2& \dots &x^2_m\\
	\vdots&\vdots& \dots&\vdots\\
	x^m_1&x^m_2& \dots &x^m_m \end{array}\right|.\]
	
\begin{thm}\emph{}

 \begin{itemize}
 \item If $m=n=l$ then there exists such an element $\mathbf{u}_0\in \mathbb{A}$ that for any $\mathbf{u}^i=\sum\limits_{j=1}^{m}x^i_je_j\in \mathbb{A}$, where $i=1,2,,3,...,m$, the equality
	\begin{equation} w_{m,d}(\mathbf{u}^1,\mathbf{u}^2,...,\mathbf{u}^m)=|\mathbf{u}^1,\mathbf{u}^2,...,\mathbf{u}^m|\mathbf{u}_0 \end{equation} holds true;
 \item if $n>m$ then $w_d(\mathbf{u}^1,\mathbf{u}^2,...,\mathbf{u}^l)=0. $
 \end{itemize}
 \end{thm}
	
	\begin{proof} It is clear that $$w(\mathbf{u}^1,\mathbf{u}^2,...,\mathbf{u}^l)=\sum\limits_{\begin{array}{l}{i_j=1,_{j=1,...,l,}}\\ s=1\end{array}}^m x^1_{i_1}x^2_{i_2}...x^l_{i_l}c_s^{(i_1,i_2,...,i_l)}e_s,$$ where $c_s^{(i_1,i_2,...,i_l)}\in \mathbb{F}$ and therefore,
$$\begin{array}{ll}
 w_{n,d}(\mathbf{u}^1,\mathbf{u}^2,...,\mathbf{u}^l)&= \sum\limits_{\begin{array}{l}{i_j=1,_{j=1,...,l,}}\\ s=1\end{array}}^m\sum\limits_{\sigma\in S_n}sgn(\sigma)x^{\sigma(1)}_{i_1}x^{\sigma(2)}_{i_2}...x^{\sigma(n)}_{i_n}x^{n+1}_{i_{n+1}}x^{n+2}_{i_{n+2}}...x^l_{i_l}c_s^{(i_1,i_2,...,i_l)}e_s\\ &= \sum\limits_{\begin{array}{l}{i_j=1,_{j=1,...,l,}}\\ s=1\end{array}}^m\left|\begin{array}{cccc}x^1_{i_1}&x^1_{i_2}&...&x^1_{i_n}\\ x^2_{i_1}&x^2_{i_2}&...&x^2_{i_n}\\
	\vdots&\vdots&...&\vdots\\
	x^n_{i_1}&x^n_{i_2}&...&x^n_{i_n} \end{array}\right|x^{n+1}_{i_{n+1}}x^{n+2}_{i_{n+2}}...x^l_{i_l}c_s^{(i_1,i_2,...,i_l)}e_s.
\end{array}$$

	The determinant in the expression above is zero whenever either two of the numbers $i_1,i_2,...,i_n$ are the same and hence, in this case we have $w_{n,d}(\mathbf{u}^1,\mathbf{u}^2,...,\mathbf{u}^l)=0$ whenever $n> m$, whereas if $m=n=l$ one has
	$$w_{m,d}(\mathbf{u}^1,\mathbf{u}^2,...,\mathbf{u}^m)=|\mathbf{u}^1,\mathbf{u}^2,...,\mathbf{u}^m|\mathbf{u}_0,\ \mbox{where}\
	\mathbf{u}_0=\sum_{s=1}^m\sum_{\sigma\in S_m}c_s^{(\sigma(1),\sigma(2),...,\sigma(m))}e_s.$$
	\end{proof}	
	\begin{cor} For any $m$-dimensional algebra $\mathbb{A}$ the following identity
	\[[w_{m,d}(\mathbf{u}^1,\mathbf{u}^2,...,\mathbf{u}^m),w_{m,d}(\mathbf{v}^1,\mathbf{v}^2,...,\mathbf{v}^m)]=0\]
	holds true.\end{cor}
	

	Further we deal only with two-dimensional algebras. Due to the theorem above in any two-dimensional algebra $\mathbb{A}$ the following identities hold true
	 \begin{equation} \label{ID} [[\mathbf{u},\mathbf{v}],[\mathbf{u}',\mathbf{v}']]=0,\ \ [\mathbf{u},\mathbf{v}]\mathbf{w}+[\mathbf{v},\mathbf{w}]\mathbf{u}+[\mathbf{w},\mathbf{u}]\mathbf{v}=0,\  \ \mathbf{w}[\mathbf{u},\mathbf{v}]+\mathbf{u}[\mathbf{v},\mathbf{w}]+\mathbf{v}[\mathbf{w},\mathbf{u}]=0 \end{equation} for all $\mathbf{u},\mathbf{v},\mathbf{w}, \mathbf{u}',\mathbf{v}'\in \mathbb{A}$. In particular, for any such algebra $(\mathbb{A}, \cdot)$ the corresponding  $(\mathbb{A}, [\cdot,\cdot])$ is a Lie algebra. Note that any three-linear identity in a two-dimensional algebra is a linear combination of the identities (\ref{ID}).
	
	Let us now observe a few simple identities for some classes of two-dimensional algebras presented above in terms of their MSC (note that these identities also hold true for the corresponding isomorphic algebras).
	
	In the cases of $A_4$, $A_5$, $A_8(\alpha_1)$ and $A_9$ one has

 \begin{center}$\begin{array}{ll} \qquad \ \ [\mathbf{u},\mathbf{v}]&=|\mathbf{u},\mathbf{v}|(\beta_2+\alpha_1-1)e_2,\\
                     \end{array}$

 $\begin{array}{ll}
 \ \ \ [\mathbf{u},\mathbf{v}]&= (3\alpha_1-2)|\mathbf{u},\mathbf{v}|e_2,\\
\end{array}$

  $\begin{array}{ll}
 [\mathbf{u},\mathbf{v}]&= (x_1 y_2-x_2 y_1) e_2
\end{array}$
\end{center}
  $\begin{array}{ll}
\mbox{and} &
\end{array}$

$$\begin{array}{ll}
 [\mathbf{u},\mathbf{v}]&=(x_1y_2-x_2y_1)e_2,
\end{array}$$
$\begin{array}{ll}
\mbox{respectively}. &
\end{array}$

   Due to $e_2^2=0$ these imply the identities \[[\mathbf{u},\mathbf{v}][\mathbf{u}',\mathbf{v}']=0 \ \mbox{and}\ [\mathbf{u},\mathbf{v},\mathbf{w}][\mathbf{u}',\mathbf{v}',\mathbf{w}']=0.\]
	
	Moreover, for $A_9$ the identity \[2[\mathbf{u},\mathbf{v}]\mathbf{w}+\mathbf{w}[\mathbf{u},\mathbf{v}]=0 \] also holds true. Indeed,

\ \qquad$\mathbf{u}\mathbf{v}=(x_1e_1+x_2e_2)(y_1e_1+y_2e_2)=\frac{x_1y_1}{3}e_1+\frac{3x_1y_1+2x_1y_2-x_2y_1}{3}e_2,$

	\ \quad$[\mathbf{u},\mathbf{v}]=(x_1y_2-x_2y_1)e_2,$

\ \ $[\mathbf{u},\mathbf{v}]\mathbf{w}=-\frac{z_1}{3}(x_1y_2-x_2y_1)e_2,$

 \ \ $\mathbf{w}[\mathbf{u},\mathbf{v}]=\frac{2z_1}{3}(x_1y_2-x_2y_1)e_2.$

Therefore, we have $$2[\mathbf{u},\mathbf{v}]\mathbf{w}+\mathbf{w}[\mathbf{u},\mathbf{v}]=0.$$
		
	For algebras $A_{10}$,\ $A_{11}$ and $A_{12}$ one has $\mathbf{u}\mathbf{v}=\mathbf{v}\mathbf{u}$. Moreover, in the case of $A_{10}$ we have
	$$\mathbf{u}\mathbf{v}=(x_1e_1+x_2e_2)(y_1e_1+y_2e_2)=(x_1y_2+x_2y_1)e_1-x_2y_2e_2,$$ $$(\mathbf{u}\mathbf{v})\mathbf{w}=((x_1y_2+x_2y_1)z_2-x_2y_2z_1)e_1+x_2y_2z_2e_2,$$
	$$\mathbf{u}(\mathbf{v}\mathbf{w})=(-x_1y_2z_2+x_2(y_1z_2+y_2z_1))e_1+x_2y_2z_2e_2,$$
	$$[\mathbf{u},\mathbf{v},\mathbf{w}]=(\mathbf{u}\mathbf{v})\mathbf{w}-\mathbf{u}(\mathbf{v}\mathbf{w})=2y_2(x_1z_2-x_2z_1)e_1$$ and hence, the identities
	\[[\mathbf{u},\mathbf{v},\mathbf{w}][\mathbf{u}',\mathbf{v}',\mathbf{w}']=0,\ \ [\mathbf{u},\mathbf{v},\mathbf{w}]+[\mathbf{v},\mathbf{w},\mathbf{u}]-[\mathbf{w},\mathbf{u},\mathbf{v}]=0\] hold true.

	 In $A_{11}$ we have
	 $\mathbf{u}\mathbf{v}=(x_1e_1+x_2e_2)(y_1e_1+y_2e_2)=(x_1y_2+x_2y_1)e_1+(x_1y_1-x_2y_2)e_2$, $$(\mathbf{u}\mathbf{v})\mathbf{w}=((x_1y_2+x_2y_1)z_2+(x_1y_1-x_2y_2)z_1)e_1+((x_1y_2+x_2y_1)z_1-(x_1y_1-x_2y_2)z_2)e_2,$$
	 $$\mathbf{u}(\mathbf{v}\mathbf{w})=(x_1(y_1z_1-y_2z_2)z_2+x_2(y_1z_2+y_2z_1))e_1+(x_1(y_1z_2+y_2z_1)-x_2(y_1z_1-y_2z_2))e_2,$$
	 $$[\mathbf{u},\mathbf{v},\mathbf{w}]=(\mathbf{u}\mathbf{v})\mathbf{w}-\mathbf{u}(\mathbf{v}\mathbf{w})=2(x_1z_2-x_2z_1)(y_2e_1-y_1e_2)$$ and
	 therefore, $$[\mathbf{u},\mathbf{v},\mathbf{w}]=-[\mathbf{w},\mathbf{v},\mathbf{u}],\ [\mathbf{u},\mathbf{v},\mathbf{w}]+[\mathbf{v},\mathbf{w},\mathbf{u}]+[\mathbf{w},\mathbf{u},\mathbf{v}]=0$$ hold true for the algebra $A_{11}$.
	
	In the case of $A_{12}$ the identities $(\mathbf{u}\mathbf{v})\mathbf{w}=\mathbf{u}(\mathbf{v}\mathbf{w})=0$ are valid.
\section{Classification of two-dimensional algebras with respect to identities}	
 Recall that in the opposite algebra $\mathbb{A}^{op}$ to an algebra $(\mathbb{A},\cdot)$ the product $*$ is defined by $\mathbf{u}*\mathbf{v}=\mathbf{v}\cdot \mathbf{u}$. Observe that an algebra $(\mathbb{A},\cdot)$ satisfies a polynomial identity if and only if its ``opposite'' $(\mathbb{A}^{op},*)$ satisfies the corresponding ``opposite'' polynomial identity, i.e., the polynomial identities of an ``opposite'' algebra can be derived from the polynomial identities of the original algebra.

 For an algebra $\mathbb{A}$ with MSC $A$ the MSC $A^{op}$ of  $\mathbb{A}^{op}$ can be easily found. In two-dimensional case it is evident that MSC of $\mathbb{A}^{op}$ with respect to the basis $\mathrm{e}$ is  $$A^{op}=\left(
	\begin{array}{cccc}
	\alpha_1 & \alpha_3 &\alpha_2 & \alpha_4 \\
	\beta_1 & \beta _3 & \beta_2 & \beta_4
	\end{array}\right),\ \mbox{provided that}\ A=\left(
	\begin{array}{cccc}
	\alpha_1 & \alpha_2 &\alpha_3 & \alpha_4 \\
	\beta_1 & \beta _2 & \beta_3 & \beta_4
	\end{array}\right)$$ is MSC of $\mathbb{A}$ with respect to $\mathrm{e}$.
	
	The main motivation to raise the question on the opposite algebras is the fact that the list of canonical representatives presented in Section 2 is not invariant with respect to the ``opposite'' operation and therefore the study of their behavior with respect to this operation is important, moreover, later on we use the result in the classification of two-dimensional algebras with respect to identities. Here are the results on ``opposite''s of two-dimensional algebras.
	
	\begin{thm} \label{Th41}

Let $\mathrm{Char}(\mathbb{F})\neq 2,3.$ The following hold true:
\begin{itemize}
\item	$(A_{1}(\alpha_1,\alpha_2,\alpha_4,\beta_1))^{op}\cong A_{1}(-\alpha_2,-\alpha_1,\beta_1,\alpha_4),$
\item $(A_{2}(\alpha_1,\beta_1,\beta_2))^{op}\cong A_{2}\left(\frac{\alpha_1}{\alpha_1+\beta_2},\frac{\beta_1}{(\alpha_1+\beta_2)\sqrt{\alpha_1+\beta_2}},\frac{1-\alpha_1}{\alpha_1+\beta_2}\right),\ \mbox{whenever}\ \alpha_1+\beta_2\neq 0,$\\
$(A_{2}(\alpha_1,\beta_1,-\alpha_1))^{op}= A_{6}(\alpha_1,\beta_1),$
\item $(A_{3}(\beta_1,\beta_2))^{op}\cong A_{3}\left(\frac{\beta_1}{\beta_2^2},\frac{1}{\beta_2}\right), \mbox{whenever}\ \beta_2\neq 0,$\\
$(A_3(\beta_1,0))^{op}=A_7(\beta_1),$
\item	$(A_{4}(\alpha_1,\beta_2))^{op}\cong A_{4}\left(\frac{\alpha_1}{\alpha_1+\beta_2},\frac{1-\alpha_1}{\alpha_1+\beta_2}\right),\ \mbox{whenever}\ \alpha_1+\beta_2\neq 0,$\\
$(A_4(\alpha_1,-\alpha_1))^{op}=A_8(\alpha_1),$
\item $(A_{5}(\alpha_1))^{op}\cong A_{5}\left(\frac{\alpha_1}{3\alpha_1-1}\right),\ \mbox{whenever}\ \alpha_1\neq \frac{1}{3},$\\
$\left(A_5\left(\frac{1}{3}\right)\right)^{op}=A_9,$
\item	$(A_6(\alpha_1,\beta_1))^{op}= A_2(\alpha_1,\beta_1,-\alpha_1),$
\item $(A_7(\beta_1))^{op}= A_3(\beta_1,0),$
\item $(A_8(\alpha_1))^{op}= A_4(\alpha_1,-\alpha_1),$
    \item $(A_9)^{op}=A_5\left(\frac{1}{3}\right),$
\item $(A_{10})^{op}= A_{10},$
\item $(A_{11})^{op}= A_{11},$
\item $(A_{12})^{op}= A_{12}.$
\end{itemize}
\end{thm}
	\begin{proof} We provide non-trivial base change that brings an algebra to its opposite unless they are equal.
	
	 Indeed,  $g(A_{1}(\alpha_1,\alpha_2,\alpha_4,\beta_1))^{op}(g^{-1} \otimes g^{-1})=\left(
	\begin{array}{cccc}
	-\alpha _2 & -\alpha _1 & 1-\alpha _1 & \beta _1 \\
	\alpha _4 & \alpha _2 & \alpha _2+1 & \alpha _1 \\
	\end{array}
	\right)=A_{1}(-\alpha_2,-\alpha_1,\beta_1,\alpha_4),$ where $g=\left(
	\begin{array}{cc}
	0 & 1 \\
	1 & 0 \\
	\end{array}
	\right)$.
\begin{itemize}
\item If $\alpha_1+\beta_2\neq 0$ then
		$g(A_{2}(\alpha_1,\beta_1,\beta_2))^{op}(g^{-1} \otimes g^{-1})= A_{2}\left(\frac{\alpha_1}{\alpha_1+\beta_2},\frac{\beta_1}{(\alpha_1+\beta_2)\sqrt{\alpha_1+\beta_2}},\frac{1-\alpha_1}{\alpha_1+\beta_2}\right),$
	 where $g=\left(
	 \begin{array}{cc}
	 \alpha_1+\beta_2 & 0 \\
	 0 & \sqrt{\alpha_1+\beta_2} \\
	 \end{array}
	 \right);$
\item If $\alpha_1+\beta_2= 0$ then $(A_{2}(\alpha_1,\beta_1,-\alpha_1))^{op}= A_{6}(\alpha_1,\beta_1).$
\end{itemize}
\begin{itemize}
\item If $\beta_2\neq 0$ then  \[g(A_{3}(\alpha_1,\beta_1, \beta_2)^{op}(g^{-1} \otimes g^{-1})=\left(
	\begin{array}{cccc}
	0 & 1 &1 & 0 \\
	\frac{\beta_1}{\beta_2^2} & \frac{1}{\beta_2} & 1& -1 \\
	\end{array}
	\right)=A_{3}\left(\frac{\beta_1}{\beta_2^2},\frac{1}{\beta_2}\right),\] where $g=\left(
	\begin{array}{cc}
	\beta_2 & 0 \\
	0 & 1 \\
	\end{array}
	\right);$
\item $(A_3(\beta_1,0))^{op}=A_7(\beta_1)$.
\end{itemize}
\begin{itemize}
\item If  $\alpha_1+\beta_2\neq 0$ then
	  \[g(A_{4}(\alpha_1,\beta_2))^{op}(g^{-1} \otimes g^{-1})=\left(
	\begin{array}{cccc}
	\frac{\alpha_1}{\alpha_1+\beta_2} & 0 & 0 & 0 \\
	0&\frac{1-\alpha_1}{\alpha_1+\beta_2} & 1-\frac{\alpha_1}{\alpha_1+\beta_2} & 0 \\
	\end{array}
	\right)= A_{4}\left(\frac{\alpha_1}{\alpha_1+\beta_2},\frac{1-\alpha_1}{\alpha_1+\beta_2}\right),\] where $g=\left(
	\begin{array}{cc}
	\alpha_1+\beta_2 & 0 \\
	0 & 1 \\
	\end{array}
	\right);$
\item $ (A_4(\alpha_1,-\alpha_1))^{op}=A_8(\alpha_1)$.
\end{itemize}
\begin{itemize}
\item If $3\alpha_1-1\neq 0$ and  $g=\left(
	\begin{array}{cc}
	3\alpha_1-1 & 0 \\
	0 & (3\alpha_1-1)^2 \\
	\end{array}
	\right)$ then
		 \[g(A_{5}(\alpha_1))^{op}(g^{-1} \otimes g^{-1})=\left(
	\begin{array}{cccc}
	\frac{\alpha_1}{3\alpha_1+1} & 0 & 0 & 0 \\
	1 & 2\frac{\alpha_1}{3\alpha_1-1}-1 &1-\frac{\alpha_1}{3\alpha_1+1}& 0 \\
	\end{array}
	\right)=A_{5}\left(\frac{\alpha_1}{3\alpha_1-1}\right);\]
\item $\left(A_5\left(\frac{1}{3}\right)\right)^{op}=A_9.$
\end{itemize}
	Now due to the equality $(A^{op})^{op}=A$ the above obtained isomorphisms and equalities imply that

 $$(A_6(\alpha_1,\beta_1))^{op}= A_2(\alpha_1,\beta_1,-\alpha_1),\ (A_7(\beta_1))^{op}= A_3(\beta_1,0),$$

  $$(A_8(\alpha_1))^{op}= A_4(\alpha_1,-\alpha_1),\ (A_9)^{op}=A_5\left(\frac{1}{3}\right).$$

   The equalities \[(A_{10})^{op}= A_{10},\
	(A_{11})^{op}= A_{11},\ (A_{12})^{op}= A_{12}\] are evident.
	\end{proof}
	\begin{cor} Let $\mathrm{Char}(\mathbb{F})\neq 2,3$.  Then, up to isomorphism, there exist only the following nontrivial two-dimensional algebras $\mathbb{A}$ with $\mathbb{A}^{op}\cong \mathbb{A}$:
\[A_1(\alpha_1,-\alpha_1,\alpha_2,\alpha_2),\ A_2(\alpha_1,\beta_1,1-\alpha_1),\ A_2(0,0,-1),\ A_3(\beta_1,-1),\ A_3(\beta_1,1),\ A_4(\alpha_1, 1-\alpha_1),\] \[ A_4(0, -1),\ A_5\left(\frac{2}{3}\right),\ A_5(0),\ A_{10},\ A_{11},\ A_{12}.\]
 \end{cor}
	
		\begin{thm} \label{Th43} Let $\mathrm{Char}(\mathbb{F})= 2$. Then the following isomorphisms and equalities hold true
\begin{itemize}
\item $(A_{1,2}(\alpha_1,\alpha_2,\alpha_4,\beta_1))^{op}\cong A_{1,2}(\alpha_2,\alpha_1,\beta_1,\alpha_4),$
\item $(A_{2,2}(\alpha_1,\beta_1,\beta_2))^{op}\cong A_{2,2}\left(\frac{\alpha_1}{\alpha_1+\beta_2},\frac{\beta_1}{(\alpha_1+\beta_2)\sqrt{\alpha_1+\beta_2}},\frac{1+\alpha_1}{\alpha_1+\beta_2}\right),\ \mbox{whenever}\ \alpha_1+\beta_2\neq 0,$\\
$(A_{2,2}(\alpha_1,\beta_1,\alpha_1))^{op}= A_{6,2}(\alpha_1,\beta_1),$
\item $(A_{3,2}(\alpha_1,\beta_2))^{op}\cong A_{3,2}\left(\frac{\alpha_1}{\alpha_1+\beta_2},\frac{1+\alpha_1}{\alpha_1+\beta_2}\right), \mbox{whenever}\ \alpha_1+\beta_2\neq 0,$\\ $(A_{3,2}(\alpha_1,\alpha_1))^{op}=A_{7,2}(\alpha_1),$
\item $(A_{4,2}(\alpha_1,\beta_2))^{op}\cong A_{4,2}\left(\frac{\alpha_1}{\alpha_1+\beta_2},\frac{1+\alpha_1}{\alpha_1+\beta_2}\right),\ \mbox{whenever}\ \alpha_1+\beta_2\neq 0,$\\ $(A_{4,2}(\alpha_1,\alpha_1))^{op}=A_{8,2}(\alpha_1),$
\item $(A_{5,2}(\alpha_1))^{op}\cong A_{5,2}\left(\frac{\alpha_1}{\alpha_1+1}\right),\ \mbox{whenever}\ \alpha_1\neq 1,$\\
		$(A_{5,2}(1))^{op}=A_{9,2},$
\item	$(A_{6,2}(\alpha_1,\beta_1))^{op}= A_{2,2}(\alpha_1,\beta_1,\alpha_1),$
 \item $(A_{7,2}(\alpha_1))^{op}= A_{3,2}(\alpha_1,\alpha_1),$
 \item $(A_{8,2}(\alpha_1))^{op}= A_{4,2}(\alpha_1,\alpha_1),$
 \item $(A_{9,2})^{op}=A_{5,2}(1),$
 \item $(A_{10,2})^{op}= A_{10,2},$
\item $(A_{11,2})^{op}= A_{11,2},\ (A_{12,2})^{op}= A_{12,2}.$
\end{itemize}
	\end{thm}
\begin{proof} The proof is similar to that of previous theorem, for example,
	$$g(A_{3,2}(\alpha_1,\beta_2))^{op}(g^{-1})^{\otimes 2}=A_{3,2}\left(\frac{\alpha_1}{\alpha_1+\beta_2},\frac{1+\alpha_1}{\alpha_1+\beta_2}\right),$$

\qquad\qquad\qquad\qquad\qquad\qquad\qquad\qquad\qquad\qquad\qquad where
	$g=\left(
	\begin{array}{cc}
	\alpha_1+\beta_2 & 0 \\
	0 & 1 \\
	\end{array}
	\right)$ and $\alpha_1+\beta_2\neq 0$.
	\end{proof}
		\begin{cor} Let $\mathrm{Char}(\mathbb{F})= 2.$ Then, up to isomorphism, there exist only the following nontrivial two-dimensional algebras with $\mathbb{A}^{op}\cong \mathbb{A}$:
		\[A_{1,2}(\alpha_1,\alpha_1,\alpha_2,\alpha_2),\ A_{2,2}(\alpha_1,\beta_1,1+\alpha_1),\ A_{3,2}(\alpha_1,1+\alpha_1),\ A_{4,2}(\alpha_1,1+\alpha_1),\ A_{5,2}(0),\ A_{10,2},\ A_{11,2},\ A_{12,2}.\]
		
	\end{cor}
	
		\begin{thm} \label{Th45} Let $\mathrm{Char}(\mathbb{F})=3$. The following isomorphisms and equalities are true
\begin{itemize}
\item $(A_{1,3}(\alpha_1,\alpha_2,\alpha_4,\beta_1))^{op}\cong A_{1,3}(-\alpha_2,-\alpha_1,\beta_1,\alpha_4),$
\item $(A_{2,3}(\alpha_1,\beta_1,\beta_2))^{op}\cong A_{2,3}\left(\frac{\alpha_1}{\alpha_1+\beta_2},\frac{\beta_1}{(\alpha_1+\beta_2)\sqrt{\alpha_1+\beta_2}},\frac{1-\alpha_1}{\alpha_1+\beta_2}\right),\ \mbox{whenever}\ \alpha_1+\beta_2\neq 0,$\\
$(A_{2,3}(\alpha_1,\beta_1,-\alpha_1))^{op}= A_{6,3}(\alpha_1,\beta_1),$
\item	$(A_{3,3}(\beta_1,\beta_2))^{op}\cong A_{3,3}\left(\frac{\beta_1}{\beta_2^2},\frac{1}{\beta_2}\right), \mbox{whenever}\ \beta_2\neq 0,$\\ $(A_{3,3}(\beta_1,0))^{op}=A_{7,3}(\beta_1),$
\item	$(A_{4,3}(\alpha_1,\beta_2))^{op}\cong A_{4,3}\left(\frac{\alpha_1}{\alpha_1+\beta_2},\frac{1-\alpha_1}{\alpha_1+\beta_2}\right),\ \mbox{whenever}\ \alpha_1+\beta_2\neq 0,$\\ $(A_{4,3}(\alpha_1,-\alpha_1))^{op}=A_{8,3}(\alpha_1),$
\item		$(A_{5,3}(\alpha_1))^{op}\cong A_{5,3}(-\alpha_1),$
\item		$(A_{6,3}(\alpha_1,\beta_1))^{op}= A_{2,3}(\alpha_1,\beta_1,-\alpha_1),$
\item $(A_{7,3}(\beta_1))^{op}= A_{3,3}(\beta_1,0),$
\item $(A_{8,3}(\alpha_1))^{op}= A_{4,3}(\alpha_1,-\alpha_1),$
\item $(A_{9,3})^{op}=A_{9,3},$
\item $(A_{10,3})^{op}= A_{10,3},$
\item $(A_{11,3})^{op}= A_{11,3},$
\item $(A_{12,3})^{op}= A_{12,3}.$
\end{itemize}
	\end{thm}
	\begin{cor} In the case of $Char(\mathbb{F})= 3$, up to isomorphism, there exist only the following nontrivial two-dimensional algebras with $\mathbb{A}^{op}\cong \mathbb{A}$:
	\[A_{1,3}(\alpha_1,-\alpha_1,\alpha_2,\alpha_2),\ A_{2,3}(\alpha_1,\beta_1,1-\alpha_1),\ A_{2,3}(0,0,-1),\ A_{3,3}(\beta_1,-1),\ A_{3,3}(\beta_1,1),\] \[A_{4,3}(\alpha_1,1-\alpha_1),\ A_{4,3}(0,-1),\ A_{5,3}(0),\ A_{9,3},\ A_{10,3},\ A_{11,3},\ A_{12,3}.\]
	
\end{cor}
	

Now we consider a set of the most important identities and classify all nontrivial two-dimensional algebras with respect to this set of identities. The set of identities mainly consists of identities which have appeared before in definitions of different classes of algebras and their ``anti'' variants. We tried not include both of an identity and its ``opposite'' versions into this set except for the case when an identity and its opposite are the same. Therefore, in calling identities the term ``Left'' is used. Of course, the use of terms ``Left'' and ``Anti'' seems to be a questionable, but we have tried to follow the tradition given in the literature earlier. From this classification one can easily derive the classification of two-dimensional algebras with a given subset of identities and due to Theorems \ref{Th41}, \ref{Th43} and \ref{Th45}, the classification of two-dimensional algebras satisfying the corresponding ``Right'' (``Opposite'') identities also can easily be obtained.

In the theorems below we give the polynomial identities denoted by $I_1, I_2, ... , I_{30}$ followed by the classification of two-dimensional algebras satisfying these identities in terms of canonical representatives $A_1 - A_{12}$.

The following notations are used:

$I$ stands for the second order identity matrix,

$i$ stands for a fixed element of $\mathbb{F}$ such that $i^2=-1$.

Elements of $\mathbb{A}$ and their position vectors are denoted by $\mathbf{u}, \mathbf{v}, \mathbf{w}$ and $u, v, w$, respectively.

\begin{thm} Let $Char(\mathbb{F})\neq 2,\ 3$. The following classification of two-dimensional algebras with respect to $I_1-I_{30}$ holds true:
\begin{itemize}	
 \item[$I_1$.] Commutativity identity $\mathbf{u}\mathbf{v}=\mathbf{v}\mathbf{u}$.

 $A_2(\alpha_1,\beta_1,1-\alpha_1)\cong A_2(\alpha_1,-\beta_1,1-\alpha_1), A_3(\beta_1,1), A_4(\alpha_1,1-\alpha_1),  A_5(\frac{2}{3}), A_{10}, A_{11}, A_{12}$.

  \item[$I_2$.] Anti-commutativity identity $\mathbf{u}\mathbf{v}=-\mathbf{v}\mathbf{u}$.
  $A_{4}(0,-1)$, $A_{12}$.

  \item[$I_3$.] Associativity identity $(\mathbf{u}\mathbf{v})\mathbf{w}=\mathbf{u}(\mathbf{v}\mathbf{w}).$

 $A_{2}\left(\frac{1}{2},0,\frac{1}{2}\right)$,
 $A_{4}\left(\frac{1}{2},\frac{1}{2}\right)$,
 $A_{4}(1,0),$
 $A_{4}\left(1,1\right),$ $A_{4}\left(\frac{1}{2},0\right)$,
  $A_{12}.$

   \item[$I_4$.] Anti-associativity identity $(\mathbf{u}\mathbf{v})\mathbf{w}=-\mathbf{u}(\mathbf{v}\mathbf{w}).$
   $A_{12}$.

  \item[$I_5$.] Well defined cube identity $\mathbf{u}^2\mathbf{u}=\mathbf{u}\mathbf{u}^2.$

 $A_{1}\left(\frac{1}{3},-\frac{1}{3},0,0\right),$ $A_{2}\left(\alpha _1,\beta _1,1-\alpha_1\right)
\cong A_{2}\left(\alpha _1,-\beta _1,1-\alpha_1\right),$
 $A_{3}\left(\beta _1,1\right),$
 $A_{4}\left(\alpha _1,1-\alpha_1\right)$,

 where $\alpha _1 \neq \frac{2}{3},$\ \
 $A_{4}\left(\alpha _1,2\alpha _1-1\right),$
 $A_{5}\left(\frac{2}{3}\right),$
 $A_{8}\left(\frac{1}{3}\right),$
 $A_{10},$
 $A_{11},$
  $A_{12}$.

   \item[$I_6$.] Half-commutativity identity $[\mathbf{u},\mathbf{v}]\mathbf{w}=\mathbf{w}[\mathbf{u},\mathbf{v}].$

 $A_2(\alpha_1,\beta_1,1-\alpha_1)\cong A_2(\alpha_1,-\beta_1,1-\alpha_1)$, $A_3(\beta_1,1)$, $A_4(\alpha_1,1-\alpha_1)$, $A_5(\frac{2}{3})$, $A_{10}$, $A_{11}$, $A_{12}$.

  \item[$I_7$.] Anti-half-commutativity identity $[\mathbf{u},\mathbf{v}]\mathbf{w}=-\mathbf{w}[\mathbf{u},\mathbf{v}].$

  $A_2(\alpha_1,\beta_1,1-\alpha_1)\cong A_2(\alpha_1,-\beta_1,1-\alpha_1)$, $A_3(\beta_1,1)$, $A_4(\alpha_1,1-\alpha_1)$, $A_4(\alpha_1,\alpha_1-1)$, $A_5(0)$, $A_5(\frac{2}{3})$, $A_{8}(\frac{1}{2})$, $A_{10}$, $A_{11}$, $A_{12}$.

   \item[$I_8$.] Mixed associativity identity $[\mathbf{u},\mathbf{v}]\mathbf{w}=\mathbf{u}[\mathbf{v},\mathbf{w}].$

  $A_2(\alpha_1,\beta_1,1-\alpha_1)\cong A_2(\alpha_1,-\beta_1,1-\alpha_1)$, $A_3(\beta_1,1)$, $ A_4(\alpha_1,1-\alpha_1)$, $A_5(\frac{2}{3})$, $A_{10}$, $A_{11}$, $A_{12}$.

    \item[$I_9$.] Anti-mixed-associativity identity $[\mathbf{u},\mathbf{v}]\mathbf{w}=-\mathbf{u}[\mathbf{v},\mathbf{w}].$

 $A_2(\alpha_1,\beta_1,1-\alpha_1)\cong A_2(\alpha_1,-\beta_1,1-\alpha_1)$, $A_3(\beta_1,1)$, $ A_4(\alpha_1,1-\alpha_1)$, $A_5(\frac{2}{3})$, $A_{10}$, $A_{11}$, $A_{12}$.

   \item[$I_{10}$.] Flexibility identity $\mathbf{u}(\mathbf{v}\mathbf{u})=(\mathbf{u}\mathbf{v})\mathbf{u}$.

 $A_2(\alpha_1,\beta_1,1-\alpha_1)\cong A_2(\alpha_1,-\beta_1,1-\alpha_1)$, $A_3(\beta_1,1)$, $ A_4(\alpha_1,1-\alpha_1)$, $A_4(\alpha_1,2\alpha_1-1),$

 where $\alpha_1\neq \frac{2}{3},$ $A_5(\frac{2}{3}),$ $A_8(\frac{1}{3}),$ $A_{10},$ $A_{11},$ $A_{12}$.

   \item[$I_{11}$.] Anti-flexibility identity $\mathbf{u}(\mathbf{v}\mathbf{u})=-(\mathbf{u}\mathbf{v})\mathbf{u}$.
 $A_{12}.$

   \item[$I_{12}$.] Mixed flexibility identity $\mathbf{u}[\mathbf{v},\mathbf{u}]=[\mathbf{u},\mathbf{v}]\mathbf{u}$.

   $A_2(\alpha_1,\beta_1,1-\alpha_1)\cong A_2(\alpha_1,-\beta_1,1-\alpha_1)$, $A_3(\beta_1,1)$, $ A_4(\alpha_1,1-\alpha_1)$, $A_4(\alpha_1,\alpha_1-1)$,

   $A_5(0)$, $A_5(\frac{2}{3})$, $A_8(\frac{1}{2})$, $ A_{10}$, $A_{11}$, $A_{12}$.

   \item[$I_{13}$.] Mixed anti-flexibility identity $\mathbf{u}[\mathbf{v},\mathbf{u}]=-[\mathbf{u},\mathbf{v}]\mathbf{u}$.

   $A_2(\alpha_1,\beta_1,1-\alpha_1)\cong A_2(\alpha_1,-\beta_1,1-\alpha_1)$, $A_3(\beta_1,1)$, $ A_4(\alpha_1,1-\alpha_1)$, $A_5(\frac{2}{3})$, $A_{10}$, $A_{11}$, $A_{12}$.

    \item[$I_{14}$.] Left Leibniz identity $\mathbf{u}(\mathbf{v}\mathbf{w})=(\mathbf{u}\mathbf{v})\mathbf{w}+\mathbf{v}(\mathbf{u}\mathbf{w})$.
       $A_4(0,-1), A_8(0), A_{12}$.

   \item[$I_{15}$.] Left anti-Leibniz identity $\mathbf{u}(\mathbf{v}\mathbf{w})=-(\mathbf{u}\mathbf{v})\mathbf{w}-\mathbf{v}(\mathbf{u}\mathbf{w})$.
       $ A_{12}$.

  \item[$I_{16}$.] Mixed left Leibniz identity $\mathbf{u}[\mathbf{v},\mathbf{w}]=[\mathbf{u},\mathbf{v}]\mathbf{w}+\mathbf{v}[\mathbf{u},\mathbf{w}]$.

  $A_2(\alpha_1,\beta_1,1-\alpha_1)\cong A_2(\alpha_1,-\beta_1,1-\alpha_1)$, $A_3(\beta_1,1)$, $ A_4(\alpha_1,1-\alpha_1)$, $A_4(\alpha_1,\alpha_1-1)$,

  $A_5(0)$, $A_5(\frac{2}{3})$, $A_8(\frac{1}{2})$, $ A_{10}$, $A_{11}$, $A_{12}$.

   \item[$I_{17}$.] Mixed anti-left Leibniz identity $\mathbf{u}[\mathbf{v},\mathbf{w}]=-[\mathbf{u},\mathbf{v}]\mathbf{w}-\mathbf{v}[\mathbf{u},\mathbf{w}]$.

   $A_2(\alpha_1,\beta_1,1-\alpha_1)\cong A_2(\alpha_1,-\beta_1,1-\alpha_1)$, $A_3(\beta_1,1)$, $ A_4(\alpha_1,1-\alpha_1)$, $A_5(\frac{2}{3})$, $A_{10}$, $A_{11}$, $A_{12}$.

    \item[$I_{18}$.] Left Poisson identity $(\mathbf{u}\mathbf{v})\mathbf{w}+(\mathbf{v}\mathbf{w})\mathbf{u}+(\mathbf{w}\mathbf{u})\mathbf{v}=0$.
         $ A_4(0,-1)$, $A_8(0)$, $A_{12}$.

     \item[$I_{19}$.] Left Jordan identity $(\mathbf{u}\mathbf{v})\mathbf{u}^2=\mathbf{u}(\mathbf{v}\mathbf{u}^2)$.

     \begin{itemize}
         \item If $Char(\mathbb{F})\neq 5$ then $A_2\left(\frac{1}{2},0,\frac{1}{2}\right),$
 $A_2\left(\frac{1}{2},0,-\frac{1}{2}\right),$
 $A_4(\alpha_1,-1+2\alpha_1),$ where $\alpha _1\neq\frac{1}{10} \left(5\pm\sqrt{5}\right),$\ \
 $A_4\left(\alpha_1,\sqrt{\alpha _1-\alpha _1^2}\right)$,
 $A_4\left(\alpha_1,-\sqrt{\alpha _1-\alpha _1^2}\right),$ where $\alpha _1\neq 0,1,$ $A_5\left(\frac{1}{10}
 \left(5-\sqrt{5}\right)\right),$ \\
  $A_5\left(\frac{1}{10} \left(5+\sqrt{5}\right)\right),$
 $A_8\left(\frac{1}{3}\right),$
  $A_8\left(\frac{1}{2}-\frac{i}{2}\right),$
 $A_8\left(\frac{1}{2}+\frac{i}{2}\right),$
 $A_{12}.$
 \item If $Char(\mathbb{F})= 5$ then $A_2\left(\frac{1}{2},0,\frac{1}{2}\right), A_2\left(\frac{1}{2},0,-\frac{1}{2}\right), A_4(\alpha_1,-1+2\alpha_1), A_4\left(\alpha_1,\sqrt{\alpha _1-\alpha _1^2}\right),$ $A_4\left(\alpha_1,-\sqrt{\alpha _1-\alpha _1^2}\right),$ where $\alpha _1\neq 0,1$, \ \ $A_8\left(\frac{1}{3}\right),$ $ A_8\left(\frac{3}{2}\right),$ $A_9,$ $A_{12}.$
\end{itemize}
   \item[$I_{20}$.] Left anti-Jordan identity $(\mathbf{u}\mathbf{v})\mathbf{u}^2=-\mathbf{u}(\mathbf{v}\mathbf{u}^2)$.
         $A_4(0,-1)$, $A_4(0,0)$, $A_{12}$.

  \item[$I_{21}$.] Mixed left Jordan identity $[\mathbf{u},\mathbf{v}]\mathbf{u}^2=\mathbf{u}[\mathbf{v},\mathbf{u}^2].$

  $A_2(\alpha_1,\beta_1,1-\alpha_1)\cong A_2(\alpha_1,-\beta_1,1-\alpha_1)$, $A_3(\beta_1,1)$, $A_4(0,-1)$,
  $A_4(0,0)$, $A_4(\alpha_1,1-\alpha_1),$

  $A_5(\frac{2}{3}),$ $A_{10},$ $A_{11},$ $A_{12}$.

    \item[$I_{22}$.] Mixed anti-left Jordan identity $[\mathbf{u},\mathbf{v}]\mathbf{u}^2=-\mathbf{u}[\mathbf{v},\mathbf{u}^2].$

    $A_2(\alpha_1,\beta_1,1-\alpha_1)\cong A_2(\alpha_1,-\beta_1,1-\alpha_1)$, $A_3(\beta_1,1)$, $A_4(0,-1),$ $A_4(0,0),$ $A_4(\alpha_1,1-\alpha_1),$

    $A_5(\frac{2}{3}),$ $A_{10},$ $A_{11},$ $A_{12}$.

       \item[$I_{23}$.] Left Malcev identity $((\mathbf{u}\mathbf{v})\mathbf{w}+(\mathbf{v}\mathbf{w})\mathbf{u}
           +(\mathbf{w}\mathbf{u})\mathbf{v})\mathbf{u}=(\mathbf{u}\mathbf{v})(\mathbf{u}\mathbf{w})
           +(\mathbf{v}(\mathbf{u}\mathbf{w}))\mathbf{u} +((\mathbf{u}\mathbf{w})\mathbf{u})\mathbf{v}$.

  $A_2(\frac{1}{2},0,\frac{1}{2})$, $A_4(0,-1)$, $A_4(\frac{1}{2},\frac{1}{2})$, $A_4(1,0)$, $A_8(0)$, $A_{12}$.

    \item[$I_{24}$.] Left anti-Malcev identity $((\mathbf{u}\mathbf{v})\mathbf{w}+(\mathbf{v}\mathbf{w})\mathbf{u}
        +(\mathbf{w}\mathbf{u})\mathbf{v})\mathbf{u}=-(\mathbf{u}\mathbf{v})(\mathbf{u}\mathbf{w})
        -(\mathbf{v}(\mathbf{u}\mathbf{w}))\mathbf{u} -((\mathbf{u}\mathbf{w})\mathbf{u})\mathbf{v}$.

          $ A_4(0,-1)$, $A_8(0)$, $A_{12}$.

    \item[$I_{25}$.] Left Zinbiel identity $ (\mathbf{u}\mathbf{v})\mathbf{w}=\mathbf{u}(\mathbf{v}\mathbf{w}+\mathbf{w}\mathbf{v})$.
      $ A_{12}$.

     \item[$I_{26}$.] Left anti-Zinbiel identity $ (\mathbf{u}\mathbf{v})\mathbf{w}=-\mathbf{u}(\mathbf{v}\mathbf{w}+\mathbf{w}\mathbf{v})$.
      $ A_{12}$.

  \item[$I_{27}$.] Left symmetric identity $[\mathbf{u},\mathbf{v},\mathbf{w}]=[\mathbf{v},\mathbf{u},\mathbf{w}]$.

 $A_2(\frac{1}{2},0,\frac{1}{2})$, $A_2(1,0,\frac{1}{2})$, $A_4(1,\beta_2)$, $A_4(\frac{1}{2},\beta_2)$, $A_5(1)$, $ A_5(\frac{1}{2})$, $A_8(0)$, $A_{12}$.

   \item[$I_{28}$.] Left anti-symmetric identity $[\mathbf{u},\mathbf{v},\mathbf{w}]=-[\mathbf{v},\mathbf{u},\mathbf{w}]$.

         $A_2(\frac{1}{2},0,\frac{1}{2})$, $A_4(\frac{1}{2},\frac{1}{2})$, $A_4(\frac{1}{2},0)$, $A_4(1,0)$, $A_4(1,1)$, $A_{12}$.

   \item[$I_{29}$.] Centro-symmetric identity $[\mathbf{u},\mathbf{v},\mathbf{w}]=[\mathbf{w},\mathbf{v},\mathbf{u}]$.

       $A_1(\alpha_1,\alpha_2, -\alpha_1-2\alpha_2,2\alpha_1+\alpha_2),$ where $\alpha_2= \frac{1}{8} \left( \pm \sqrt{32 \alpha_1-15}-8 \alpha_1-1\right)$,
        $A_2\left(\frac{1}{2},0,\frac{1}{2}\right)$,\\
         $A_4\left(\alpha_1,\frac{\alpha_1-\sqrt{12\alpha_1-7\alpha_1^2-4}}{2}\right)$,\\ $A_4\left(\alpha_1,\frac{\alpha_1+\sqrt{12\alpha_1-7\alpha_1^2-4}}{2}\right)$, $A_5\left(\frac{1}{2}\right)$, $A_5(1)$, $A_8\left(\frac{3-\sqrt{7}i}{8}\right)$, $A_8\left(\frac{3+\sqrt{7}i}{8}\right)$, $A_{12}$.

        \item[$I_{30}$.] Centro-anti-symmetric identity $[\mathbf{u},\mathbf{v},\mathbf{w}]=-[\mathbf{w},\mathbf{v},\mathbf{u}]$.

  $A_2(\alpha_1,\beta_1,1-\alpha_1)$, $A_3(\beta_1,1)$, $A_4(\alpha_1,1-\alpha_1)$, $A_4(\alpha_1,2\alpha_1-1)$, $A_5(\frac{2}{3})$, $A_8(\frac{1}{3})$, $A_{10}$, $A_{11}$, $A_{12}$.
  \end{itemize}
 \end{thm}
 \begin{proof}
Proofs of the cases $I_1-I_{30}$ are similar. The basic idea behind the proof of the theorem is just to convert the given identity in terms of MSC $A$ as a matrix equation and solve the corresponding system of polynomials equations
  for each of $A_1-A_{12}$ canonical representatives. Here we provide the systems of equations. The solutions to them are a bit technical, but once we substitute the structure constants into the system it is much simplified and can be managed combining manual and computer calculations. We illustrate the technique in two specific cases (the cases $I_{18}$ and $I_{23}$).

 The cases of commutativity and anticommutativity are clear where the identities are equivalent to the fact that the second and third columns of $A$ to be the same and differ in sign as vectors, respectively.
The proof of the cases $I_5$ and $I_{19}$ are given in \cite{B4} and \cite{B3}, respectively.

The identities $I_3$ and $I_4$ in terms of matrices are expressed by $A(A\otimes I)=A(I\otimes A)$ and $A(A\otimes I)=-A(I\otimes A),$ respectively.\\

\begin{itemize}
 \item $I_6$ case. $ A(A\otimes I)((u\otimes v-v\otimes u)\otimes w)=A(I\otimes A)(w\otimes(u\otimes v-v\otimes u))$.\\
  The corresponding system is given as below

   $\begin{array}{rr}
 	\alpha _2 \beta _2-\alpha _3 \beta _2-\alpha _2 \beta _3+\alpha _3 \beta _3=0,&\\
 	\alpha _2^2-2 \alpha _3 \alpha _2+\alpha _3^2=0,&\\
 	\beta _2^2-2 \beta _3 \beta _2+\beta _3^2 =0.&\end{array}$.
 	
 \item $I_7$ case. $A(A\otimes I)((u\otimes v-v\otimes u)\otimes w)=-A(I\otimes A)(w\otimes(u\otimes v-v\otimes u)).$\\
 In this case the system is as follows

  $\begin{array}{rr}
 2 \alpha_1 \alpha _2+\beta _2 \alpha _2-\beta _3 \alpha _2-2 \alpha _1 \alpha _3+\alpha _3 \beta _2-\alpha _3 \beta _3=0,&\\
 \alpha _2 \beta _2-\alpha _3 \beta _2+2 \beta _4 \beta _2+\alpha _2 \beta _3-\alpha _3 \beta _3-2 \beta _3 \beta _4 =0,&\\
 \alpha _2^2-\alpha _3^2+2 \alpha _4 \beta _2-2 \alpha _4 \beta _3 =0,&\\
 \beta _2^2-\beta _3^2+2 \alpha _2 \beta _1-2 \alpha _3 \beta _1 =0.&
 \end{array}$
  \item $I_8$ case. $ A(A\otimes I)((u\otimes v-v\otimes u)\otimes w)=A(I\otimes A)(u\otimes(v\otimes w-w\otimes v))$.\\
   The identity is written in terms of structure constants as follows

  $\begin{array}{rr}
  2 \alpha _1 \alpha _2+\beta _2 \alpha _2-\beta _3 \alpha _2-2 \alpha _1 \alpha _3+\alpha _3 \beta _2-\alpha _3 \beta _3 =0,&\\
  \alpha _2 \beta _2-\alpha _3 \beta _2+2 \beta _4 \beta _2+\alpha _2 \beta _3-\alpha _3 \beta _3-2 \beta _3 \beta _4=0, &\\
 \alpha _1 \alpha _2+\beta _2 \alpha _2-\beta _3 \alpha _2-\alpha _1 \alpha _3=0,&\\
  \alpha _2^2-\alpha _3 \alpha _2+\alpha _4 \beta _2-\alpha _4 \beta _3 =0,&\\
 \alpha _1 \alpha _2-\alpha _1 \alpha _3+\alpha _3 \beta _2-\alpha _3 \beta _3 =0,&\\
 \alpha _2^2-\alpha _3^2+2 \alpha _4 \beta _2-2 \alpha _4 \beta _3 =0,&\\
 \alpha _3^2-\alpha _2 \alpha _3-\alpha _4 \beta _2+\alpha _4 \beta _3 =0,&\\
 \beta _2^2-\beta _3 \beta _2-\alpha _2 \beta _1-\alpha _3 \beta _1 =0,&\\
 \beta _2^2-\beta _3^2+2 \alpha _2 \beta _1-2 \alpha _3 \beta _1=0,&\\
 \alpha _2 \beta _2-\alpha _3 \beta _2+\beta _4 \beta _2-\beta _3 \beta _4=0, &\\
 \beta _3^2-\beta _2 \beta _3-\alpha _2 \beta _1+\alpha _3 \beta _1 =0,&\\
  \alpha _2 \beta _3-\alpha _3 \beta _3-\beta _4 \beta _3+\beta _2 \beta _4 =0.&
 \end{array}$

 \item $I_9$ case. $ A(A\otimes I)((u\otimes v-v\otimes u)\otimes w)=-A(I\otimes A)(u\otimes(v\otimes w-w\otimes v))$.\\
 Here is the corresponding system of equations
 \begin{center}
  $\begin{array}{rr}
 \alpha _1 \alpha _2+\beta _2 \alpha _2-\beta _3 \alpha _2-\alpha _1 \alpha _3 =0,&
 \alpha _2 \beta _2-\alpha _3 \beta _2-\alpha _2 \beta _3+\alpha _3 \beta _3 =0,\\
 \alpha _2^2-\alpha _3 \alpha _2+\alpha _4 \beta _2-\alpha _4 \beta _3 =0,&
 \alpha _1 \alpha _2-\alpha _1 \alpha _3+\alpha _3 \beta _2-\alpha _3 \beta _3=0, \\
 \beta _2^2-\beta _3 \beta _2+\alpha _2 \beta _1-\alpha _3 \beta _1 =0,&

 \alpha _3^2-\alpha _2 \alpha _3-\alpha _4 \beta _2+\alpha _4 \beta _3=0, \\

 \alpha _2 \beta _2-\alpha _3 \beta _2+\beta _4 \beta _2-\beta _3 \beta _4 =0,&
 \beta _3^2-\beta _2 \beta _3-\alpha _2 \beta _1+\alpha _3 \beta _1 =0,\\

 \alpha _2 \beta _2-\alpha _3 \beta _2-\alpha _2 \beta _3+\alpha _3 \beta _3 =0,&
 \alpha _2 \beta _3-\alpha _3 \beta _3-\beta _4 \beta _3+\beta _2 \beta _4 =0,\\
 \alpha _2^2-2 \alpha _3 \alpha _2+\alpha _3^2 =0,&
 \beta _2^2+2 \beta _3 \beta _2+\beta _3^2 =0.
 \end{array}$
\end{center}
 \item $I_{10}$ case. $(A(I\otimes A)-A(A\otimes
 I))(u\otimes v\otimes u)=0$.

 The system of equations is\\ $\begin{array}{rr}
 \alpha _3 \beta _2-\beta _4 \beta _2-\alpha _2 \beta _3+\beta _3 \beta _4 =0,
  &
 \beta _2^2-\alpha _1 \beta _2-\beta _3^2+\alpha _2 \beta _1-\alpha _3 \beta _1+\alpha _1 \beta _3=0,\\
 \alpha _2 \beta _2-\alpha _3 \beta _3 =0,&
  \alpha _2^2-\beta _4 \alpha _2-\alpha _3^2+\alpha _4 \beta _2-\alpha _4 \beta _3+\alpha _3 \beta _4 =0,\\
 \alpha _4 \beta _2-\alpha _4 \beta _3=0,&

 \alpha _1 \alpha _2-\beta _3 \alpha _2-\alpha _1 \alpha _3+\alpha _3 \beta _2=0,\\
 \beta _1 \beta _2-\beta _1 \beta _3 =0,& \alpha _3 \alpha _4-\alpha _2 \alpha _4 =0
 , \\
 \alpha _2 \beta _1-\alpha _3 \beta _1=0.&
 \end{array}$

 \item $I_{11}$ case. $(A(I\otimes A)+A(A\otimes
 I))(u\otimes v\otimes u)=0$.

 In this case the system of equations is given as follows \\ $\begin{array}{rr}
 \alpha _2^2+\beta _4 \alpha _2+\alpha _3^2+2 \alpha _1 \alpha _4+\alpha _4 \beta _2+\alpha _4 \beta _3+\alpha _3 \beta _4=0,& \\
 \beta _2^2+\alpha _1 \beta _2+\beta _3^2+\alpha _2 \beta _1+\alpha _3 \beta _1+\alpha _1 \beta _3+2 \beta _1 \beta _4 =0,&\\
 2 \alpha _1 \alpha _2+\beta _2 \alpha _2+2 \alpha _1 \alpha _3+2 \alpha _4 \beta _1+\alpha _3 \beta _3 =0,&\\
 2 \alpha _4 \beta _1+\alpha _2 \beta _2+\alpha _3 \beta _3+2 \beta _2 \beta _4+2 \beta _3 \beta _4 =0,&\\
 \alpha _1 \alpha _2+\beta _3 \alpha _2+\alpha _1 \alpha _3+\alpha _3 \beta _2 =0,&\\
 \alpha _3 \beta _2+\beta _4 \beta _2+\alpha _2 \beta _3+\beta _3 \beta _4 =0,&\\

 2 \alpha _2 \alpha _3+\alpha _4 \beta _2+\alpha _4 \beta _3=0, &\\
 \alpha _2 \alpha _4+\alpha _3 \alpha _4+2 \beta _4 \alpha _4 =0,&\\
 2 \alpha _1 \beta _1+\beta _2 \beta _1+\beta _3 \beta _1 =0, &\\
 \alpha _2 \beta _1+\alpha _3 \beta _1+2 \beta _2 \beta _3 =0,&\\

 2 \alpha _1^2+\alpha _2 \beta _1+\alpha _3 \beta _1 =0,&\\
 2 \beta _4^2+\alpha _4 \beta _2+\alpha _4 \beta _3 =0.&
 \end{array}.$

 \item $I_{12}$ case. $ A(A\otimes I)((u\otimes v-v\otimes u)\otimes u)=A(I\otimes A)(u\otimes(v\otimes u-u\otimes v)).$

  The identity can be written in terms of structure constant as the following system of equations\\ $\begin{array}{rr}
 2 \alpha _1 \alpha _2+\beta _2 \alpha _2-\beta _3 \alpha _2-2 \alpha _1 \alpha _3+\alpha _3 \beta _2-\alpha _3 \beta _3 =0,&\\
 \alpha _2 \beta _2-\alpha _3 \beta _2+2 \beta _4 \beta _2+\alpha _2 \beta _3-\alpha _3 \beta _3-2 \beta _3 \beta _4 =0.& \\
 \alpha _2^2-\alpha _3^2+2 \alpha _4 \beta _2-2 \alpha _4 \beta _3 =0,&\\
 \beta _2^2-\beta _3^2+2 \alpha _2 \beta _1-2 \alpha _3 \beta _1 =0,&

 \end{array}$

 \item $I_{13}$ case. $ A(A\otimes I)((u\otimes v-v\otimes u)\otimes u)=-A(I\otimes A)(u\otimes(v\otimes u-u\otimes v)).$\\
      The system of equations is\\  $\begin{array}{rrr}
  \alpha _2 \beta _2-\alpha _3 \beta _2-\alpha _2 \beta _3+\alpha _3 \beta _3=0, &
  \alpha _2^2-2 \alpha _3 \alpha _2+\alpha _3^2 =0,&
    \beta _2^2-2 \beta _3 \beta _2+\beta _3^2 =0.

  \end{array}$

%

 \item $I_{14}$ case. $(A(I\otimes A)-A(A\otimes I))(u\otimes v\otimes w)=A(I\otimes A)(v\otimes u\otimes w).$\\
 Here is the corresponding system of equations\\
 $\begin{array}{rr}
 \beta _3^2-\alpha _1 \beta _3+\beta _2 \beta _3+2 \alpha _3 \beta _1-\beta _1 \beta _4 =0
 	,&
\alpha _1^2+\alpha _3 \beta _1 =0,\\
\alpha _2^2+\alpha _3 \alpha _2-\beta _4 \alpha _2-\alpha _1 \alpha _4+2 \alpha _4 \beta _2 =0
 	,&

 	\alpha _4 \beta _1+\beta _3 \beta _4 =0,\\
 	\alpha _1 \alpha _3+\beta _3 \alpha _3-\alpha _4 \beta _1+\alpha _2 \beta _3 =0,&

 	\alpha _1 \beta _1+\beta _3 \beta _1 =0,\\

 	\alpha _1 \alpha _2-\beta _3 \alpha _2+\alpha _4 \beta _1+\alpha _3 \beta _2 =0

 ,&
 	\alpha _2 \alpha _4+\beta _4 \alpha _4 =0,\\
 \alpha _2 \beta _1-\alpha _3 \beta _1+\beta _4 \beta _1+\alpha _1 \beta _3 =0
 	,&
 	\alpha _1 \beta _2+\beta _1 \beta _4 =0, \\
 \alpha _4 \beta _1-\alpha _2 \beta _2-\alpha _2 \beta _3-\beta _2 \beta _4 =0
 	,&
 	\alpha _1 \alpha _2+\alpha _4 \beta _1 =0,\\
 \alpha _4 \beta _1+\alpha _3 \beta _2-\alpha _2 \beta _3+\beta _3 \beta _4 =0
 	,&
 	\alpha _1 \alpha _4+\alpha _3 \beta _4 =0,\\
 \alpha _1 \alpha _4-\beta _2 \alpha _4+\beta _3 \alpha _4+\alpha _2 \beta _4 =0
 	,&
 	\beta _4^2+\alpha _4 \beta _2 =0.
 \end{array}$

 \item $I_{15}$ case. $(A(I\otimes A)+A(A\otimes I))(u\otimes v\otimes w)=-A(I\otimes A)(v\otimes u\otimes w).$\\
 The system of equations is given below\\
 $\begin{array}{rr}
 	\alpha _1 \alpha _2+\beta _3 \alpha _2+2 \alpha _1 \alpha _3+\alpha _4 \beta _1+\alpha _3 \beta _2 =0
 , &
 2 \alpha _2 \alpha _3+\alpha _1 \alpha _4+\alpha _4 \beta _2+\alpha _4 \beta _3+\alpha _2 \beta _4 =0

  ,\\
 \alpha _2 \beta _1+\alpha _3 \beta _1+\beta _4 \beta _1+\alpha _1 \beta _3+2 \beta _2 \beta _3 =0

 	, & \alpha _2^2+\alpha _3 \alpha _2+\beta _4 \alpha _2+\alpha _1 \alpha _4+2 \alpha _4 \beta _2 =0	
 	,\\
 \beta _3^2+\alpha _1 \beta _3+\beta _2 \beta _3+2 \alpha _3 \beta _1+\beta _1 \beta _4 =0
 	,& \alpha _4 \beta _1+\alpha _3 \beta _2+\alpha _2 \beta _3+2 \beta _2 \beta _4+\beta _3 \beta _4=0
 	,\\
 	3 \alpha _1 \alpha _3+\beta _3 \alpha _3+\alpha _4 \beta _1+\alpha _2 \beta _3 =0

 , &
 	\alpha _2 \alpha _4+2 \alpha _3 \alpha _4+3 \beta _4 \alpha _4 =0,\\
 	2 \alpha _3^2+\beta _4 \alpha _3+\alpha _1 \alpha _4+2 \alpha _4 \beta _3 =0

 , &
 	2 \beta _2^2+\alpha _1 \beta _2+2 \alpha _2 \beta _1+\beta _1 \beta _4 =0,\\
 	3 \alpha _1^2+2 \alpha _2 \beta _1+\alpha _3 \beta _1=0, &
 	\alpha _4 \beta _1+\alpha _2 \beta _2+\alpha _2 \beta _3+3 \beta _2 \beta _4 =0,\\
 	3 \alpha _1 \beta _1+2 \beta _2 \beta _1+\beta _3 \beta _1 =0,&
 	3 \alpha _1 \alpha _2+2 \beta _2 \alpha _2+\alpha _4 \beta _1 =0, \\
 	\alpha _4 \beta _1+2 \alpha _3 \beta _3+3 \beta _3 \beta _4 =0,&
 	3 \beta _4^2+\alpha _4 \beta _2+2 \alpha _4 \beta _3 =0.
 \end{array}$

 \item $I_{16}$ case. $A(I\otimes A)(u\otimes (v\otimes w-w\otimes v))=A(A\otimes I)((u\otimes v-v\otimes u)\otimes w)+A(I\otimes A)(v\otimes (u\otimes w-w\otimes u)).$
 The system of equations is\\
 $\begin{array}{rr}
 	2 \alpha _1 \alpha _2+\beta _2 \alpha _2-\beta _3 \alpha _2-2 \alpha _1 \alpha _3+\alpha _3 \beta _2-\alpha _3 \beta _3 =0,&\\
 \alpha _2 \beta _2-\alpha _3 \beta _2+2 \beta _4 \beta _2+\alpha _2 \beta _3-\alpha _3 \beta _3-2 \beta _3 \beta _4 =0,&\\
 	\alpha _2^2-\alpha _3^2+2 \alpha _4 \beta _2-2 \alpha _4 \beta _3 =0,&\\
 	\beta _2^2-\beta _3^2+2 \alpha _2 \beta _1-2 \alpha _3 \beta _1 =0.&
 \end{array}$

 \item $I_{17}$ case. $A(I\otimes A)(u\otimes (v\otimes w-w\otimes v))$

\hfill $=-A(A\otimes I)((u\otimes v-v\otimes u)\otimes w)-A(I\otimes A)(v\otimes (u\otimes w-w\otimes u)).$\\
 In this case we have the system of equations\\
 $\begin{array}{rr}
 \alpha _2 \beta _2-\alpha _3 \beta _2+2 \beta _4 \beta _2+\alpha _2 \beta _3-\alpha _3 \beta _3-2 \beta _3 \beta _4 =0
 	,&
 	\alpha _2 \beta _2-\alpha _3 \beta _2-\alpha _2 \beta _3+\alpha _3 \beta _3 =0,\\
 2 \alpha _1 \alpha _2+\beta _2 \alpha _2-\beta _3 \alpha _2-2 \alpha _1 \alpha _3+\alpha _3 \beta _2-\alpha _3 \beta _3 =0
 	,&
 	\alpha _2^2-\alpha _3^2+2 \alpha _4 \beta _2-2 \alpha _4 \beta _3 =0 ,\\
  \alpha _3^2- \alpha _2 \alpha _3- \alpha _4 \beta _2+ \alpha _4 \beta _3 =0,&
 \beta _2^2-\beta _3^2+2 \alpha _2 \beta _1-2 \alpha _3 \beta _1=0
 ,\\
 	 \beta _2^2- \beta _3 \beta _2+ \alpha _2 \beta _1- \alpha _3 \beta _1 =0,&
 	\beta _2^2-2 \beta _3 \beta _2+\beta _3^2 =0,\\
 	 \alpha _1 \alpha _2+ \beta _2 \alpha _2- \beta _3 \alpha _2- \alpha _1 \alpha _3 =0,&
 	
 	\alpha _2^2-2 \alpha _3 \alpha _2+\alpha _3^2 =0,\\
 	 \alpha _2 \beta _3- \alpha _3 \beta _3- \beta _4 \beta _3+ \beta _2 \beta _4 =0.&
 \end{array}$
 \item $I_{18}$ case. The left Poisson identity $(\mathbf{u}\mathbf{v})\mathbf{w}+(\mathbf{v}\mathbf{w})\mathbf{u}+(\mathbf{w}\mathbf{u})\mathbf{v}=0$ is equivalent to $ A(A\otimes I)(u\otimes v\otimes w+v\otimes w\otimes u+w\otimes u\otimes v)=0$ that produces the system of equations


\begin{equation}\label{SELP}\begin{array}{rr}
 2 \alpha _1 \alpha _2+\alpha _1 \alpha _3+\alpha _4 \beta _1+\alpha _3 \beta _2+\alpha _3 \beta _3 &=0,\\
 \alpha _2^2+\alpha _3 \alpha _2+\alpha _1 \alpha _4+\alpha _4 \beta _2+\alpha _4 \beta _3+\alpha _3 \beta _4 &=0,\\
 \beta _3^2+\beta _2 \beta _3+\alpha _2 \beta _1+\alpha _3 \beta _1+\alpha _1 \beta _2+\beta _1 \beta _4 &=0,\\
 \alpha _4 \beta _1+\alpha _2 \beta _2+\alpha _3 \beta _2+\beta _2 \beta _4+2 \beta _3 \beta _4 &=0,\\
 3 \beta _4^2+3 \alpha _4 \beta _2 &=0. \\
 3 \alpha _1^2+3 \alpha _3 \beta _1&=0, \\
 3 \alpha _2 \alpha _4+3 \beta _4 \alpha _4&=0, \\
 3 \alpha _1 \beta _1+3 \beta _3 \beta _1 &=0,\\
\end{array}\end{equation}

For $A_1$ the system (\ref{SELP}) becomes
\begin{equation*}
\begin{array}{rr}
 3 \alpha _1^2+3 \alpha _2 \beta _1+3 \beta _1 &=0,\\
 \alpha _2 \alpha _1-\alpha _1+\alpha _2+\alpha _4 \beta _1+1&=0, \\
 \alpha _1^2-3 \alpha _1+\alpha _2 \beta _1+\beta _1+1 &=0,\\
 \alpha _2 \alpha _1-\alpha _1-2 \alpha _2+\alpha _4 \beta _1 &=0,\\
 3 \alpha _2^2-3 \alpha _1 \alpha _4 &=0, \\
  \alpha _2^2-\alpha _1 \alpha _4+\alpha _4&=0, \\
 3 \beta _1 &=0.\\
\end{array}\end{equation*} The last equation gives $\beta_1=0$ that produces $\alpha_1=0$ due to the first equation and we get a contradiction in the equation 3.

For $A_2$ the system (\ref{SELP}) has the form
\begin{equation*}
\begin{array}{rr}
\alpha _1^2-2 \alpha _1+\beta _2+1  &=0,\\
 \beta _2+1 &=0, \\
  3 \beta _2 &=0,\\
 3 \alpha _1^2  &=0,\\
 \beta _1  &=0.\\
\end{array}
\end{equation*} The equations 2 and 3 immediately lead to a contradiction.

For $A_3$ due to the system (\ref{SELP}) we have
\begin{equation*}\begin{array}{rr}
\beta _1+\beta _2+1&=0, \\
  \beta _2+1 &=0,\\
   \beta _2-2 &=0,\\
  3 \beta _1 &=0,\\
  1&=0. \\
\end{array}
\end{equation*} which is a contradiction.

For $A_4$ the system (\ref{SELP}) becomes
\begin{equation*}
\begin{array}{rr}
  \alpha _1^2-2 \alpha _1+\beta _2+1 &=0,\\
  3 \alpha _1^2 &=0.\\
\end{array}\end{equation*} from where we get $\alpha _1=0$ and $\beta _2=-1.$

For $A_5$, $A_6$ and $A_7$ the system (\ref{SELP}) gives

$\begin{array}{rr}
 3 \alpha _1^2&=0, \\
  \alpha _1^2 &=0,\\
  3 &=0.\\
\end{array}$ \quad
$\begin{array}{rr}
1-\alpha _1 &=0,\\
 \alpha _1^2 &=0,\\
 \beta _1 &=0.\\
 \end{array}$ \ \
and \ \
$\begin{array}{cc}
 \beta_1&=0, \\
 1 &=0.\\
\end{array},$
respectively. Which are contradictions.

Substituting the structure constants $A_8$ into the system (\ref{SELP}) we get
\begin{equation*}
\alpha _1^2=0.
\end{equation*} hence, $\alpha _1=0.$

For the structure constants of $ A_9, \ A_{10},\ A_{11}$ the system (\ref{SELP}) gives contradictions.

The structure constants of $A_{12}$ satisfy the system of equations (\ref{SELP}).

 \item $I_{19}$ case.  $(A(A\otimes A)-A(I\otimes A(I\otimes A)))(u\otimes v\otimes u\otimes u)=0.$\\
 More detailed proof for this case can be seen in \cite{B3}.
 In this case the system of equations has the following form\\
 $\begin{array}{rr}
 	\alpha _4 \alpha _1^2+\alpha _2^2 \alpha _1+2 \alpha _3^2 \alpha _1+2 \alpha _2 \alpha _3 \alpha _1+\alpha _4 \beta _2 \alpha _1+2 \alpha _4 \beta _3 \alpha _1+\alpha _3 \beta _4 \alpha _1+\alpha _4 \beta _2^2&\\
 	+\alpha _2 \alpha _4 \beta _1+\alpha _3 \alpha _4 \beta _1 +\alpha _2^2 \beta _2+\alpha _3^2 \beta _2+\alpha _2 \alpha _3 \beta _2+2 \alpha _2^2 \beta _3+\alpha _2 \alpha _3 \beta _3+\alpha _4 \beta _2 \beta _3&\\
 	+2 \alpha _4 \beta _1 \beta _4+\alpha _2 \beta _2 \beta _4+\alpha _2 \beta _3 \beta _4&=0, \\
 \end{array}$

 $\begin{array}{rr}
 	\beta _1 \alpha _3^2+\beta _2^2 \alpha _3+\beta _3^2 \alpha _3+\alpha _2 \beta _1 \alpha _3+\alpha _1 \beta _3 \alpha _3+\beta _2 \beta _3 \alpha _3+\beta _1 \beta _4 \alpha _3+2 \alpha _2 \beta _3^2&\\
 	+\beta _1 \beta _4^2+2 \alpha _1 \alpha _4 \beta _1 +\alpha _4 \beta _1 \beta _2+\alpha _1 \alpha _2 \beta _3+\alpha _4 \beta _1 \beta _3+\alpha _2 \beta _2 \beta _3+2 \beta _2^2 \beta _4&\\
 	+\beta _3^2 \beta _4+2 \alpha _2 \beta _1 \beta _4+\alpha _1 \beta _2 \beta _4+2 \beta _2 \beta _3 \beta _4 &=0,\\
 	\beta _2^3+\alpha _1 \beta _2^2+\beta _3 \beta _2^2+2 \alpha _2 \beta _1 \beta _2+2 \alpha _3 \beta _1 \beta _2+\alpha _1 \beta _3 \beta _2+2 \beta _1 \beta _4 \beta _2+\alpha _1 \beta _3^2&\\
 	+2 \alpha _1 \alpha _2 \beta _1+3 \alpha _1 \alpha _3 \beta _1 +\alpha _1^2 \beta _3+3 \alpha _2 \beta _1 \beta _3+\alpha _3 \beta _1 \beta _3+\alpha _1 \beta _1 \beta _4+2 \beta _1 \beta _3 \beta _4 &=0,\\
 	\alpha _3^3+\alpha _2 \alpha _3^2+\beta _4 \alpha _3^2+2 \alpha _1 \alpha _4 \alpha _3+2 \alpha _4 \beta _2 \alpha _3+2 \alpha _4 \beta _3 \alpha _3+\alpha _2 \beta _4 \alpha _3+\alpha _2 \beta _4^2&\\
 	+2 \alpha _1 \alpha _2 \alpha _4+\alpha _2 \alpha _4 \beta _2 +3 \alpha _2 \alpha _4 \beta _3+\alpha _2^2 \beta _4+\alpha _1 \alpha _4 \beta _4+3 \alpha _4 \beta _2 \beta _4+2 \alpha _4 \beta _3 \beta _4 &=0,\\

 	\beta _1 \alpha _2^2+\beta _2^2 \alpha _2+2 \alpha _3 \beta _1 \alpha _2+3 \beta _2 \beta _3 \alpha _2+2 \beta _1 \beta _4^2+\alpha _3^2 \beta _1+\alpha _1 \alpha _4 \beta _1+2 \alpha _4 \beta _1 \beta _2&\\
 	+\alpha _1 \alpha _3 \beta _3+2 \alpha _4 \beta _1 \beta _3+2 \alpha _3 \beta _2 \beta _3+2 \beta _2^2 \beta _4+2 \alpha _1 \beta _3 \beta _4+2 \beta _2 \beta _3 \beta _4 &=0,\\
 2 \alpha _4 \alpha _1^2+2 \alpha _3^2 \alpha _1+2 \alpha _2 \alpha _3 \alpha _1+2 \alpha _2 \beta _4 \alpha _1+\alpha _4 \beta _2^2+\alpha _4 \beta _3^2+2 \alpha _2 \alpha _4 \beta _1+2 \alpha _3 \alpha _4 \beta _1&\\
 	+2 \alpha _2 \alpha _3 \beta _2+\alpha _3^2 \beta _3 +3 \alpha _2 \alpha _3 \beta _3+2 \alpha _4 \beta _2 \beta _3+\alpha _4 \beta _1 \beta _4+\alpha _2 \beta _2 \beta _4 &=0,\\
 	\beta _3 \alpha _3^2+2 \alpha _4 \beta _1 \alpha _3+\alpha _2 \beta _3 \alpha _3+2 \beta _3 \beta _4 \alpha _3+\alpha _4 \beta _2^2+\alpha _4 \beta _3^2+4 \beta _2 \beta _4^2+2 \beta _3 \beta _4^2&\\
  	+2 \alpha _2 \alpha _4 \beta _1+4 \alpha _4 \beta _2 \beta _3+\alpha _4 \beta _1 \beta _4+\alpha _2 \beta _2 \beta _4+2 \alpha _2 \beta _3 \beta _4 &=0,\\
%
 	2 \alpha _2 \alpha _1^2+4 \alpha _3 \alpha _1^2+\alpha _4 \beta _1 \alpha _1+2 \alpha _2 \beta _2 \alpha _1+2 \alpha _2 \beta _3 \alpha _1+\alpha _3 \beta _3 \alpha _1+\alpha _2 \beta _2^2+\alpha _2^2 \beta _1&\\
 	+\alpha _3^2 \beta _1+4 \alpha _2 \alpha _3 \beta _1+2 \alpha _4 \beta _1 \beta _2+2 \alpha _4 \beta _1 \beta _3+\alpha _2 \beta _2 \beta _3 &=0,\\
 	\alpha _4 \beta _3^2+\beta _4^2 \beta _3+\alpha _1 \alpha _4 \beta _3+\alpha _2 \beta _4 \beta _3+\beta _2 \beta _4^2+\alpha _3 \alpha _4 \beta _1+\alpha _4 \beta _1 \beta _4+\alpha _3 \beta _2 \beta _4 &=0,\\
 	\alpha _2 \alpha _1^2+\alpha _3 \alpha _1^2+\alpha _4 \beta _1 \alpha _1+\alpha _3 \beta _2 \alpha _1+\alpha _2 \beta _3 \alpha _1+\alpha _2^2 \beta _1+\alpha _4 \beta _1 \beta _2+\alpha _2 \beta _1 \beta _4 &=0,\\
 	\beta _1 \alpha _4^2+2 \alpha _1 \alpha _3 \alpha _4+\alpha _3 \beta _3 \alpha _4+\beta _2 \beta _4 \alpha _4+\beta _3 \beta _4 \alpha _4+2 \alpha _2 \alpha _3 \beta _4 &=0,\\
 \alpha _4 \beta _1^2+\alpha _1 \alpha _2 \beta _1+\alpha _1 \alpha _3 \beta _1+\alpha _2 \beta _2 \beta _1+2 \beta _2 \beta _4 \beta _1+2 \alpha _1 \beta _2 \beta _3&=0, \\
 	2 \beta _1 \alpha _1^2+2 \beta _1 \beta _2 \alpha _1+\beta _1 \beta _3 \alpha _1+\alpha _2 \beta _1^2+\beta _1 \beta _2^2+\beta _1^2 \beta _4 &=0,\\
 	\alpha _4 \alpha _3^2+2 \alpha _4 \beta _4 \alpha _3+\alpha _1 \alpha _4^2+2 \alpha _4 \beta _4^2+\alpha _4^2 \beta _3+\alpha _2 \alpha _4 \beta _4&=0, \\	
 	2 \beta _4^3+\alpha _4 \beta _2 \beta _4+3 \alpha _4 \beta _3 \beta _4+\alpha _4^2 \beta _1+\alpha _3 \alpha _4 \beta _3 &=0,\\
 2 \alpha _1^3+3 \alpha _2 \beta _1 \alpha _1+\alpha _3 \beta _1 \alpha _1+\alpha _4 \beta _1^2+\alpha _2 \beta _1 \beta _2 &=0.
 \end{array}$

 \item $I_{20}$ case. $(A(A\otimes A)+A(I\otimes A(I\otimes A)))(u\otimes v\otimes u\otimes u)=0$.

 Here is the system of equations\\
 $\begin{array}{rr}
 	\alpha _4 \alpha _1^2+\alpha _2^2 \alpha _1+2 \alpha _3^2 \alpha _1+2 \alpha _2 \alpha _3 \alpha _1+\alpha _4 \beta _2 \alpha _1+2 \alpha _4 \beta _3 \alpha _1+\alpha _3 \beta _4 \alpha _1+\alpha _4 \beta _2^2&\\
 	+\alpha _2 \alpha _4 \beta _1+\alpha _3 \alpha _4 \beta _1 +\alpha _2^2 \beta _2+\alpha _3^2 \beta _2+\alpha _2 \alpha _3 \beta _2+2 \alpha _2^2 \beta _3+\alpha _2 \alpha _3 \beta _3+\alpha _4 \beta _2 \beta _3&\\
 	+2 \alpha _4 \beta _1 \beta _4+\alpha _2 \beta _2 \beta _4+\alpha _2 \beta _3 \beta _4&=0, \\
 \beta _1 \alpha _3^2+\beta _2^2 \alpha _3+\beta _3^2 \alpha _3+\alpha _2 \beta _1 \alpha _3+\alpha _1 \beta _3 \alpha _3+\beta _2 \beta _3 \alpha _3+\beta _1 \beta _4 \alpha _3+2 \alpha _2 \beta _3^2&\\
 	+\beta _1 \beta _4^2+2 \alpha _1 \alpha _4 \beta _1 +\alpha _4 \beta _1 \beta _2+\alpha _1 \alpha _2 \beta _3+\alpha _4 \beta _1 \beta _3+\alpha _2 \beta _2 \beta _3+2 \beta _2^2 \beta _4&\\
 	+\beta _3^2 \beta _4+2 \alpha _2 \beta _1 \beta _4+\alpha _1 \beta _2 \beta _4+2 \beta _2 \beta _3 \beta _4 &=0,\\
%
 	\alpha _3^3+\alpha _2 \alpha _3^2+\beta _4 \alpha _3^2+2 \alpha _1 \alpha _4 \alpha _3+2 \alpha _4 \beta _2 \alpha _3+2 \alpha _4 \beta _3 \alpha _3+\alpha _2 \beta _4 \alpha _3+\alpha _2 \beta _4^2&\\
 	+2 \alpha _1 \alpha _2 \alpha _4+\alpha _2 \alpha _4 \beta _2 +3 \alpha _2 \alpha _4 \beta _3+\alpha _2^2 \beta _4+\alpha _1 \alpha _4 \beta _4+3 \alpha _4 \beta _2 \beta _4+2 \alpha _4 \beta _3 \beta _4 &=0,\\

 	\beta _2^3+\alpha _1 \beta _2^2+\beta _3 \beta _2^2+2 \alpha _2 \beta _1 \beta _2+2 \alpha _3 \beta _1 \beta _2+\alpha _1 \beta _3 \beta _2+2 \beta _1 \beta _4 \beta _2+\alpha _1 \beta _3^2&\\
 	+2 \alpha _1 \alpha _2 \beta _1+3 \alpha _1 \alpha _3 \beta _1 +\alpha _1^2 \beta _3+3 \alpha _2 \beta _1 \beta _3+\alpha _3 \beta _1 \beta _3+\alpha _1 \beta _1 \beta _4+2 \beta _1 \beta _3 \beta _4 &=0,\\
%
 	2 \alpha _4 \alpha _1^2+2 \alpha _3^2 \alpha _1+2 \alpha _2 \alpha _3 \alpha _1+2 \alpha _2 \beta _4 \alpha _1+\alpha _4 \beta _2^2+\alpha _4 \beta _3^2+2 \alpha _2 \alpha _4 \beta _1+2 \alpha _3 \alpha _4 \beta _1&\\
 	+2 \alpha _2 \alpha _3 \beta _2+\alpha _3^2 \beta _3 +3 \alpha _2 \alpha _3 \beta _3+2 \alpha _4 \beta _2 \beta _3+\alpha _4 \beta _1 \beta _4+\alpha _2 \beta _2 \beta _4 &=0,\\
 	 	\beta _1 \alpha _2^2+\beta _2^2 \alpha _2+2 \alpha _3 \beta _1 \alpha _2+3 \beta _2 \beta _3 \alpha _2+2 \beta _1 \beta _4^2+\alpha _3^2 \beta _1+\alpha _1 \alpha _4 \beta _1+2 \alpha _4 \beta _1 \beta _2&\\
 	+\alpha _1 \alpha _3 \beta _3+2 \alpha _4 \beta _1 \beta _3+2 \alpha _3 \beta _2 \beta _3+2 \beta _2^2 \beta _4+2 \alpha _1 \beta _3 \beta _4+2 \beta _2 \beta _3 \beta _4 &=0,\\
 	\beta _3 \alpha _3^2+2 \alpha _4 \beta _1 \alpha _3+\alpha _2 \beta _3 \alpha _3+2 \beta _3 \beta _4 \alpha _3+\alpha _4 \beta _2^2+\alpha _4 \beta _3^2+4 \beta _2 \beta _4^2+2 \beta _3 \beta _4^2&\\
 	+2 \alpha _2 \alpha _4 \beta _1+4 \alpha _4 \beta _2 \beta _3+\alpha _4 \beta _1 \beta _4+\alpha _2 \beta _2 \beta _4+2 \alpha _2 \beta _3 \beta _4 &=0,\\
 \end{array}$

 $\begin{array}{rr}
 	2 \alpha _2 \alpha _1^2+4 \alpha _3 \alpha _1^2+\alpha _4 \beta _1 \alpha _1+2 \alpha _2 \beta _2 \alpha _1+2 \alpha _2 \beta _3 \alpha _1+\alpha _3 \beta _3 \alpha _1+\alpha _2 \beta _2^2+\alpha _2^2 \beta _1&\\
 	+\alpha _3^2 \beta _1+4 \alpha _2 \alpha _3 \beta _1+2 \alpha _4 \beta _1 \beta _2+2 \alpha _4 \beta _1 \beta _3+\alpha _2 \beta _2 \beta _3 &=0,\\
 	\alpha _2 \alpha _1^2+\alpha _3 \alpha _1^2+\alpha _4 \beta _1 \alpha _1+\alpha _3 \beta _2 \alpha _1+\alpha _2 \beta _3 \alpha _1+\alpha _2^2 \beta _1+\alpha _4 \beta _1 \beta _2+\alpha _2 \beta _1 \beta _4 &=0,\\
 	\alpha _4 \beta _3^2+\beta _4^2 \beta _3+\alpha _1 \alpha _4 \beta _3+\alpha _2 \beta _4 \beta _3+\beta _2 \beta _4^2+\alpha _3 \alpha _4 \beta _1+\alpha _4 \beta _1 \beta _4+\alpha _3 \beta _2 \beta _4 &=0,\\

 	\beta _1 \alpha _4^2+2 \alpha _1 \alpha _3 \alpha _4+\alpha _3 \beta _3 \alpha _4+\beta _2 \beta _4 \alpha _4+\beta _3 \beta _4 \alpha _4+2 \alpha _2 \alpha _3 \beta _4 &=0,\\
 	\alpha _4 \alpha _3^2+2 \alpha _4 \beta _4 \alpha _3+\alpha _1 \alpha _4^2+2 \alpha _4 \beta _4^2+\alpha _4^2 \beta _3+\alpha _2 \alpha _4 \beta _4&=0,\\
 	  	2 \beta _1 \alpha _1^2+2 \beta _1 \beta _2 \alpha _1+\beta _1 \beta _3 \alpha _1+\alpha _2 \beta _1^2+\beta _1 \beta _2^2+\beta _1^2 \beta _4 &=0,\\
 	\alpha _4 \beta _1^2+\alpha _1 \alpha _2 \beta _1+\alpha _1 \alpha _3 \beta _1+\alpha _2 \beta _2 \beta _1+2 \beta _2 \beta _4 \beta _1+2 \alpha _1 \beta _2 \beta _3&=0, \\
 	2 \alpha _1^3+3 \alpha _2 \beta _1 \alpha _1+\alpha _3 \beta _1 \alpha _1+\alpha _4 \beta _1^2+\alpha _2 \beta _1 \beta _2 &=0,\\
 	2 \beta _4^3+\alpha _4 \beta _2 \beta _4+3 \alpha _4 \beta _3 \beta _4+\alpha _4^2 \beta _1+\alpha _3 \alpha _4 \beta _3 &=0.

 \end{array}$

 $\begin{array}{rr}

   \end{array}$

  \item $I_{21}$ case. $ A(A\otimes A)((u\otimes v-v\otimes u)\otimes u^{\otimes 2})=A(I\otimes A(I\otimes A))(u\otimes v\otimes u^{\otimes 2})-A(I\otimes A(A\otimes I))(u^{\otimes 3}\otimes v).$\\
  The system of equations is given as follows\\
 $\begin{array}{rrr}
 		\alpha _1 \alpha _2^2+\beta _2 \alpha _2^2+\beta _3 \alpha _2^2+\alpha _3 \beta _2 \alpha _2-\alpha _3 \beta _3 \alpha _2+\alpha _1 \beta _4 \alpha _2+\beta _2 \beta _4 \alpha _2-\beta _3 \beta _4 \alpha _2-\alpha _1 \alpha _3^2&& \\
 	 	+2 \alpha _4 \beta _2^2 -2 \alpha _4 \beta _3^2-2 \alpha _3^2 \beta _3-\alpha _1 \alpha _3 \beta _4 &=0, &\\
 2 \beta _1 \alpha _2^2+2 \beta _2^2 \alpha _2-\beta _3^2 \alpha _2+\alpha _1 \beta _3 \alpha _2+\beta _2 \beta _3 \alpha _2-\alpha _3 \beta _3^2-2 \alpha _3^2 \beta _1-\alpha _1 \alpha _3 \beta _3-\alpha _3 \beta _2 \beta _3 && \\
 	+\beta _2^2 \beta _4-\beta _3^2 \beta _4+\alpha _1 \beta _2 \beta _4-\alpha _1 \beta _3 \beta _4  &=0, &\\
 	2 \alpha _1 \alpha _2^2+2 \beta _2 \alpha _2^2+\alpha _1 \alpha _3 \alpha _2+\alpha _3 \beta _2 \alpha _2-3 \alpha _3 \beta _3 \alpha _2-3 \alpha _1 \alpha _3^2+\alpha _4 \beta _2^2-\alpha _4 \beta _3^2+\alpha _3^2 \beta _2&&\\+\alpha _1 \alpha _4 \beta _2-\alpha _3^2 \beta _3-\alpha _1 \alpha _4 \beta _3  &=0, &\\
 	\alpha _3^3-\alpha _2^2 \alpha _3+2 \alpha _1 \alpha _4 \alpha _3-2 \alpha _4 \beta _2 \alpha _3+2 \alpha _4 \beta _3 \alpha _3+\alpha _2 \beta _4 \alpha _3-2 \alpha _1 \alpha _2 \alpha _4-2 \alpha _2 \alpha _4 \beta _2 && \\
 	+2 \alpha _2 \alpha _4 \beta _3-\alpha _2^2 \beta _4-\alpha _4 \beta _2 \beta _4+\alpha _4 \beta _3 \beta _4 &=0, &\\
 	\beta _2^3-\beta _3^2 \beta _2+2 \alpha _2 \beta _1 \beta _2-2 \alpha _3 \beta _1 \beta _2+\alpha _1 \beta _3 \beta _2+2 \beta _1 \beta _4 \beta _2-\alpha _1 \beta _3^2+\alpha _1 \alpha _2 \beta _1 && \\
 	-\alpha _1 \alpha _3 \beta _1+2 \alpha _2 \beta _1 \beta _3-2 \alpha _3 \beta _1 \beta _3-2 \beta _1 \beta _3 \beta _4  &=0, &\\
 \beta _2^2 \alpha _2+3 \beta _2 \beta _3 \alpha _2+\beta _1 \beta _4 \alpha _2-\alpha _3 \beta _2^2-2 \alpha _3 \beta _3^2-\alpha _3^2 \beta _1+\alpha _2^2 \beta _1-\alpha _3 \beta _2 \beta _3&&\\+3 \beta _2^2 \beta _4-2 \beta _3^2 \beta _4-\alpha _3 \beta _1 \beta _4-\beta _2 \beta _3 \beta _4 &=0, &\\
 	\alpha _3 \alpha _1^2-\alpha _2 \beta _2 \alpha _1-\alpha _2 \beta _3 \alpha _1+2 \alpha _3 \beta _3 \alpha _1-\alpha _2 \beta _2^2+\alpha _2 \beta _3^2-\alpha _1^2 \alpha _2-\alpha _2^2 \beta _1&&\\+\alpha _3^2 \beta _1-2 \alpha _4 \beta _1 \beta _2+2 \alpha _4 \beta _1 \beta _3 &=0, & \\
 	\beta _3 \alpha _2^2+2 \alpha _4 \beta _1 \alpha _2+2 \beta _2 \beta _4 \alpha _2-\beta _3 \beta _4 \alpha _2+\alpha _4 \beta _2^2-\alpha _4 \beta _3^2+\beta _2 \beta _4^2-\beta _3 \beta _4^2&&\\-2 \alpha _3 \alpha _4 \beta _1-\alpha _3^2 \beta _3-\alpha _3 \beta _3 \beta _4  &=0, &\\
 	2 \alpha _2 \alpha _1^2-2 \alpha _3 \alpha _1^2+\alpha _2 \beta _2 \alpha _1+\alpha _3 \beta _2 \alpha _1-\alpha _2 \beta _3 \alpha _1-\alpha _3 \beta _3 \alpha _1+\alpha _2^2 \beta _1-\alpha _2 \alpha _3 \beta _1&&\\+\alpha _4 \beta _1 \beta _2-\alpha _4 \beta _1 \beta _3 &=0, &\\
 	\alpha _4 \beta _3^2+2 \beta _4^2 \beta _3-\alpha _4 \beta _2 \beta _3-\alpha _2 \beta _4 \beta _3+\alpha _3 \beta _4 \beta _3-2 \beta _2 \beta _4^2-\alpha _2 \alpha _4 \beta _1+\alpha _3 \alpha _4 \beta _1&&\\-\alpha _2 \beta _2 \beta _4+\alpha _3 \beta _2 \beta _4 &=0, & \\

 	\beta _4 \alpha _2^2+\alpha _1 \alpha _4 \alpha _2-\alpha _1 \alpha _3 \alpha _4+\alpha _3 \alpha _4 \beta _2-\alpha _3 \alpha _4 \beta _3-\alpha _3^2 \beta _4+2 \alpha _4 \beta _2 \beta _4-2 \alpha _4 \beta _3 \beta _4 &=0, & \\
 	
 	\alpha _1 \beta _2^2+\alpha _2 \beta _1 \beta _2-\alpha _3 \beta _1 \beta _2+\beta _1 \beta _4 \beta _2-\alpha _1 \beta _3^2+2 \alpha _1 \alpha _2 \beta _1-2 \alpha _1 \alpha _3 \beta _1-\beta _1 \beta _3 \beta _4  &=0, &\\

 	\alpha _2 \alpha _4 \beta _3-\alpha _3 \alpha _4 \beta _3-\alpha _4 \beta _4 \beta _3+\alpha _4 \beta _2 \beta _4 &=0. &\\
 \alpha _1 \alpha _2 \beta _1-\alpha _1 \alpha _3 \beta _1+\alpha _2 \beta _2 \beta _1-\alpha _2 \beta _3 \beta _1  &=0,& \\
 \alpha _4 \alpha _3^2-\alpha _2 \alpha _4 \alpha _3-\alpha _4^2 \beta _2+\alpha _4^2 \beta _3  &=0, &\\
 	\alpha _2 \beta _1^2-\alpha _3 \beta _1^2+\beta _2^2 \beta _1-\beta _2 \beta _3 \beta _1 &=0. &
 \end{array}$

 \item $I_{22}$ case. $A(A\otimes A)((u\otimes v-v\otimes u)\otimes u^{\otimes 2})$

 \hfill $=-A(I\otimes A(I\otimes A))(u\otimes v\otimes u^{\otimes 2})+A(I\otimes A(A\otimes I))(u^{\otimes 3}\otimes v).$

 In this case the system of equations is\\
 $\begin{array}{rr}
 	\alpha _1 \alpha _2^2+\beta _2 \alpha _2^2+\beta _3 \alpha _2^2-\alpha _3 \beta _2 \alpha _2-3 \alpha _3 \beta _3 \alpha _2-\alpha _1 \beta _4 \alpha _2-\beta _2 \beta _4 \alpha _2+\beta _3 \beta _4 \alpha _2-\alpha _1 \alpha _3^2&\\
 +2 \alpha _3^2 \beta _2+\alpha _1 \alpha _3 \beta _4 &=0,\\
 	2 \alpha _3 \beta _2^2-\beta _4 \beta _2^2-3 \alpha _2 \beta _3 \beta _2-\alpha _3 \beta _3 \beta _2+\alpha _1 \beta _4 \beta _2+\alpha _2 \beta _3^2+\alpha _3 \beta _3^2+\alpha _1 \alpha _2 \beta _3-\alpha _1 \alpha _3 \beta _3&\\+\beta _3^2 \beta _4-\alpha _1 \beta _3 \beta _4 &=0,\\
 	\beta _2^2 \alpha _2-2 \beta _3^2 \alpha _2+\beta _2 \beta _3 \alpha _2-\beta _1 \beta _4 \alpha _2-\alpha _3 \beta _2^2-\alpha _3^2 \beta _1+\alpha _2^2 \beta _1+\alpha _3 \beta _2 \beta _3-\beta _2^2 \beta _4&\\+\alpha _3 \beta _1 \beta _4+\beta _2 \beta _3 \beta _4 &=0,\\
 	2 \beta _3 \alpha _2^2-\alpha _1 \alpha _3 \alpha _2-\alpha _3 \beta _2 \alpha _2-\alpha _3 \beta _3 \alpha _2+\alpha _1 \alpha _3^2+\alpha _4 \beta _2^2-\alpha _4 \beta _3^2+\alpha _3^2 \beta _2-\alpha _1 \alpha _4 \beta _2&\\-\alpha _3^2 \beta _3+\alpha _1 \alpha _4 \beta _3 &=0,\\
%
 	\beta _3 \alpha _2^2-\beta _3 \beta _4 \alpha _2+\alpha _4 \beta _2^2+\alpha _4 \beta _3^2-\beta _2 \beta _4^2+\beta _3 \beta _4^2-\alpha _3^2 \beta _3-2 \alpha _4 \beta _2 \beta _3&\\+2 \alpha _3 \beta _2 \beta _4-\alpha _3 \beta _3 \beta _4&=0, \\
 	\alpha _3 \alpha _1^2+\alpha _2 \beta _2 \alpha _1-2 \alpha _3 \beta _2 \alpha _1+\alpha _2 \beta _3 \alpha _1+\alpha _2 \beta _2^2-\alpha _2 \beta _3^2-\alpha _1^2 \alpha _2-\alpha _2^2 \beta _1&\\-\alpha _3^2 \beta _1+2 \alpha _2 \alpha _3 \beta _1&=0, \\
 	\alpha _4 \beta _3^2-\alpha _4 \beta _2 \beta _3+\alpha _2 \beta _4 \beta _3-\alpha _3 \beta _4 \beta _3-\alpha _2 \alpha _4 \beta _1+\alpha _3 \alpha _4 \beta _1-\alpha _2 \beta _2 \beta _4+\alpha _3 \beta _2 \beta _4 &=0,\\
 \beta _1 \alpha _2^2-\alpha _3 \beta _1 \alpha _2-\alpha _1 \beta _2 \alpha _2+\alpha _1 \beta _3 \alpha _2+\alpha _1 \alpha _3 \beta _2+\alpha _4 \beta _1 \beta _2-\alpha _1 \alpha _3 \beta _3-\alpha _4 \beta _1 \beta _3 &=0,\\
 	
 	\alpha _3^3-\alpha _2^2 \alpha _3-\alpha _2 \beta _4 \alpha _3-2 \alpha _2 \alpha _4 \beta _2+2 \alpha _2 \alpha _4 \beta _3+\alpha _2^2 \beta _4+\alpha _4 \beta _2 \beta _4-\alpha _4 \beta _3 \beta _4&=0, \\
 \beta _2^3-\beta _3^2 \beta _2-\alpha _1 \beta _3 \beta _2+\alpha _1 \beta _3^2-\alpha _1 \alpha _2 \beta _1+\alpha _1 \alpha _3 \beta _1+2 \alpha _2 \beta _1 \beta _3-2 \alpha _3 \beta _1 \beta _3 &=0,\\	
 	\beta _4 \alpha _2^2+\alpha _1 \alpha _4 \alpha _2-2 \alpha _3 \beta _4 \alpha _2-\alpha _1 \alpha _3 \alpha _4+\alpha _3 \alpha _4 \beta _2-\alpha _3 \alpha _4 \beta _3+\alpha _3^2 \beta _4 &=0,\\

 	\alpha _1 \beta _2^2-\alpha _2 \beta _1 \beta _2+\alpha _3 \beta _1 \beta _2-2 \alpha _1 \beta _3 \beta _2-\beta _1 \beta _4 \beta _2+\alpha _1 \beta _3^2+\beta _1 \beta _3 \beta _4 &=0,\\
 	\alpha _1 \alpha _2 \beta _1-\alpha _1 \alpha _3 \beta _1+\alpha _2 \beta _2 \beta _1-\alpha _2 \beta _3 \beta _1&=0, \\

 \alpha _2 \alpha _4 \beta _3-\alpha _3 \alpha _4 \beta _3-\alpha _4 \beta _4 \beta _3+\alpha _4 \beta _2 \beta _4 &=0.\\
 	
 	\alpha _2 \beta _1^2-\alpha _3 \beta _1^2+\beta _2^2 \beta _1-\beta _2 \beta _3 \beta _1&=0, \\
 \alpha _4 \alpha _3^2-\alpha _2 \alpha _4 \alpha _3-\alpha _4^2 \beta _2+\alpha _4^2 \beta _3 &=0,\\	
 \end{array}$

 \item $I_{23}$ case. $A(A(A\otimes I)\otimes I)((u\otimes v\otimes w+v\otimes w\otimes u+w\otimes u\otimes v)\otimes u)$\\ $=A(A\otimes A)(u\otimes v\otimes u\otimes w)+A(A(I\otimes A)\otimes I)(v\otimes u\otimes w\otimes u)+A(A(A\otimes I)\otimes I)(u\otimes w\otimes u\otimes v).$
 The identity is equivalent to the following system of equations\\
 $\begin{array}{rrr}
 	\beta _1 \alpha _2^2+\alpha _1 \beta _2 \alpha _2+\beta _1 \beta _4 \alpha _2+\alpha _3 \beta _2^2-\alpha _3 \beta _3^2-2 \beta _1 \beta _4^2-\alpha _3^2 \beta _1-\alpha _1 \alpha _3 \beta _2-\alpha _4 \beta _1 \beta _2&&\\
 	+\alpha _4 \beta _1 \beta _3+\beta _3^2 \beta _4+\alpha _3 \beta _1 \beta _4-2 \alpha _1 \beta _3 \beta _4+\beta _2 \beta _3 \beta _4 &=0,&\\
 	2 \alpha _4 \alpha _1^2-\alpha _2^2 \alpha _1-\alpha _2 \alpha _3 \alpha _1-\alpha _4 \beta _2 \alpha _1-\alpha _4 \beta _3 \alpha _1+2 \alpha _2 \beta _4 \alpha _1+\alpha _4 \beta _2^2-\alpha _4 \beta _3^2&&\\
 	-\alpha _2 \alpha _4 \beta _1+\alpha _3 \alpha _4 \beta _1+\alpha _2^2 \beta _2-\alpha _3^2 \beta _2+\alpha _3 \beta _2 \beta _4-\alpha _3 \beta _3 \beta _4 &=0,&\\
 	\alpha _2 \beta _2^2+\beta _4 \beta _2^2-\alpha _1 \alpha _2 \beta _2-\alpha _1 \alpha _3 \beta _2-\alpha _4 \beta _1 \beta _2-2 \alpha _3 \beta _3 \beta _2+\beta _3 \beta _4 \beta _2+\alpha _2 \beta _3^2+2 \alpha _1 \alpha _4 \beta _1&&\\
 	+\alpha _4 \beta _1 \beta _3-2 \beta _3^2 \beta _4+\alpha _2 \beta _1 \beta _4-\alpha _3 \beta _1 \beta _4 & =0,&\\
 	\beta _1 \alpha _3^2-\beta _2^2 \alpha _3+\beta _3^2 \alpha _3+\alpha _2 \beta _1 \alpha _3+\beta _2 \beta _3 \alpha _3-\beta _1 \beta _4^2-2 \alpha _1 \alpha _2 \beta _2-\alpha _4 \beta _1 \beta _2+\alpha _4 \beta _1 \beta _3&&\\
 	+\alpha _2 \beta _2 \beta _3+\beta _3^2 \beta _4-\alpha _2 \beta _1 \beta _4-\alpha _1 \beta _2 \beta _4&=0,& \\
 	2 \alpha _1 \alpha _2^2-\beta _3 \alpha _2^2-\alpha _1 \alpha _3 \alpha _2-\alpha _4 \beta _1 \alpha _2+2 \alpha _3 \beta _2 \alpha _2-\alpha _1 \alpha _3^2+\alpha _3 \alpha _4 \beta _1+\alpha _1 \alpha _4 \beta _2-\alpha _3^2 \beta _3&&\\
 	-\alpha _1 \alpha _4 \beta _3-2 \alpha _4 \beta _1 \beta _4+\alpha _3 \beta _2 \beta _4+\alpha _3 \beta _3 \beta _4&=0,&\\

 	\alpha _4 \alpha _1^2+\alpha _4 \beta _3 \alpha _1+\alpha _3 \beta _4 \alpha _1-\alpha _1 \alpha _2^2-\alpha _4 \beta _2^2-\alpha _2 \alpha _4 \beta _1+\alpha _3 \alpha _4 \beta _1-\alpha _2^2 \beta _2+\alpha _3^2 \beta _2&& \\
 	-\alpha _2 \alpha _3 \beta _2-\alpha _2 \alpha _3 \beta _3-\alpha _4 \beta _2 \beta _3+2 \alpha _3 \beta _3 \beta _4 &=0,&\\

 	\alpha _3 \alpha _1^2+\alpha _4 \beta _1 \alpha _1-2 \alpha _2 \beta _2 \alpha _1+\alpha _3 \beta _2 \alpha _1+\alpha _3 \beta _3 \alpha _1-\alpha _3 \beta _2^2+\alpha _3 \beta _3^2-\alpha _1^2 \alpha _2+\alpha _3^2 \beta _1&& \\
 	-2 \alpha _2 \alpha _3 \beta _1-\alpha _4 \beta _1 \beta _2+\alpha _3 \beta _1 \beta _4 &=0,&\\
 \end{array}$

   $\begin{array}{rrr}	
 	2 \alpha _1 \alpha _2^2-\beta _3 \alpha _2^2-\alpha _1 \alpha _3 \alpha _2+2 \alpha _4 \beta _1 \alpha _2+\alpha _3 \beta _2 \alpha _2-\alpha _1 \alpha _3^2-\alpha _3^2 \beta _2+\alpha _1 \alpha _4 \beta _2-\alpha _3^2 \beta _3&&\\
 	-\alpha _1 \alpha _4 \beta _3+\alpha _4 \beta _1 \beta _4-\alpha _3 \beta _3 \beta _4 &=0,&\\
 	\alpha _2^3-\alpha _3^2 \alpha _2+\alpha _1 \alpha _4 \alpha _2+2 \alpha _4 \beta _2 \alpha _2-2 \alpha _4 \beta _3 \alpha _2+\alpha _3 \beta _4 \alpha _2-\alpha _1 \alpha _3 \alpha _4+\alpha _4^2 \beta _1-\alpha _3 \alpha _4 \beta _3&&\\
 	-\alpha _3^2 \beta _4+\alpha _4 \beta _2 \beta _4-\alpha _4 \beta _3 \beta _4&=0, &\\
 	\beta _3^3-\beta _2^2 \beta _3-2 \alpha _2 \beta _1 \beta _3+2 \alpha _3 \beta _1 \beta _3+\alpha _1 \beta _2 \beta _3+\beta _1 \beta _4 \beta _3+\alpha _4 \beta _1^2-\alpha _1 \beta _2^2-\alpha _1 \alpha _2 \beta _1&&\\
 	+\alpha _1 \alpha _3 \beta _1-\alpha _2 \beta _1 \beta _2-\beta _1 \beta _2 \beta _4 &=0,&\\
 	\alpha _2 \beta _2^2+\alpha _3 \beta _2^2+\beta _4 \beta _2^2+\alpha _1 \alpha _2 \beta _2-\alpha _3 \beta _3 \beta _2+\beta _3 \beta _4 \beta _2+\alpha _2 \beta _3^2-\alpha _1 \alpha _4 \beta _1-2 \alpha _4 \beta _1 \beta _3&& \\
 	-2 \beta _3^2 \beta _4+\alpha _2 \beta _1 \beta _4-\alpha _3 \beta _1 \beta _4 &=0,&\\
 	\beta _2 \alpha _2^2+\beta _2 \beta _4 \alpha _2+\alpha _4 \beta _2^2+\beta _2 \beta _4^2-\beta _3 \beta _4^2-\alpha _3 \alpha _4 \beta _1-\alpha _3^2 \beta _2+\alpha _1 \alpha _4 \beta _2-2 \alpha _4 \beta _2 \beta _3&&\\
 	+\alpha _4 \beta _1 \beta _4+\alpha _3 \beta _2 \beta _4-2 \alpha _3 \beta _3 \beta _4 &=0,&\\
   \beta _4 \alpha _2^2-\alpha _1 \alpha _4 \alpha _2+2 \alpha _4 \beta _2 \alpha _2+\alpha _3 \beta _4 \alpha _2-2 \alpha _3 \beta _4^2+\alpha _1 \alpha _3 \alpha _4-3 \alpha _3 \alpha _4 \beta _2+\alpha _3 \alpha _4 \beta _3&&\\
 	-2 \alpha _1 \alpha _4 \beta _4+3 \alpha _4 \beta _2 \beta _4-\alpha _4 \beta _3 \beta _4 &=0,&\\
 	 	3 \alpha _2 \alpha _1^2-3 \alpha _3 \alpha _1^2+2 \alpha _4 \beta _1 \alpha _1+\alpha _3 \beta _2 \alpha _1-2 \alpha _2 \beta _3 \alpha _1-\alpha _3 \beta _3 \alpha _1-\alpha _2^2 \beta _1-2 \alpha _3^2 \beta _1&& \\
 	+3 \alpha _2 \alpha _3 \beta _1-\alpha _4 \beta _1 \beta _2+\alpha _4 \beta _1 \beta _3 &=0,&\\
%
 	\alpha _2^3+\alpha _3 \alpha _2^2-\beta _4 \alpha _2^2-\alpha _1 \alpha _4 \alpha _2+2 \alpha _4 \beta _2 \alpha _2+\alpha _4 \beta _3 \alpha _2-\alpha _3 \beta _4^2-\alpha _1 \alpha _3 \alpha _4-\alpha _3 \alpha _4 \beta _3&&\\
 	-\alpha _1 \alpha _4 \beta _4+\alpha _4 \beta _3 \beta _4&=0,& \\
 	\beta _3^3-\alpha _1 \beta _3^2+\beta _2 \beta _3^2+\alpha _2 \beta _1 \beta _3+2 \alpha _3 \beta _1 \beta _3-\beta _1 \beta _4 \beta _3+\alpha _1 \alpha _2 \beta _1-\alpha _1^2 \beta _2&\\-\alpha _2 \beta _1 \beta _2-\alpha _1 \beta _1 \beta _4-\beta _1 \beta _2 \beta _4 &=0,&\\

 	2 \beta _2 \alpha _1^2-\beta _3^2 \alpha _1+\alpha _2 \beta _1 \alpha _1-3 \alpha _3 \beta _1 \alpha _1-\beta _2 \beta _3 \alpha _1+2 \beta _1 \beta _4 \alpha _1-\alpha _2 \beta _1 \beta _2+3 \alpha _3 \beta _1 \beta _2&\\
 	-2 \alpha _3 \beta _1 \beta _3-\beta _1 \beta _2 \beta _4+\beta _1 \beta _3 \beta _4 &=0,&\\

 	2 \alpha _4 \beta _2^2+3 \beta _4^2 \beta _2-3 \alpha _4 \beta _3 \beta _2+\alpha _2 \beta _4 \beta _2-\alpha _3 \beta _4 \beta _2+\alpha _4 \beta _3^2-3 \beta _3 \beta _4^2-\alpha _2 \alpha _4 \beta _1&&\\
 	+\alpha _3 \alpha _4 \beta _1-2 \alpha _4 \beta _1 \beta _4+2 \alpha _2 \beta _3 \beta _4&=0,& \\

 	\beta _2 \alpha _2^2+\alpha _3 \beta _2 \alpha _2+\alpha _4 \beta _2^2-\alpha _4 \beta _3^2-\alpha _2 \alpha _4 \beta _1-\alpha _3 \alpha _4 \beta _1-\alpha _4 \beta _1 \beta _4+\alpha _3 \beta _2 \beta _4&=0, &\\
 	\alpha _3 \beta _3^2+\alpha _3^2 \beta _1-\alpha _2^2 \beta _1-\alpha _1 \alpha _4 \beta _1+\alpha _1 \alpha _3 \beta _2-\alpha _4 \beta _1 \beta _2-\alpha _4 \beta _1 \beta _3+\alpha _3 \beta _2 \beta _3 &=0,&\\
 	2 \alpha _4 \alpha _2^2-\alpha _3 \alpha _4 \alpha _2-\alpha _1 \alpha _4^2+2 \alpha _4^2 \beta _2-\alpha _4^2 \beta _3-\alpha _3 \alpha _4 \beta _4&=0,& \\
 	\alpha _2 \beta _1^2-2 \alpha _3 \beta _1^2+\beta _4 \beta _1^2-2 \beta _3^2 \beta _1+\alpha _1 \beta _2 \beta _1+\beta _2 \beta _3 \beta _1 &=0,&\\

 	\alpha _4 \beta _1^2+2 \alpha _1 \alpha _2 \beta _1-2 \alpha _1 \alpha _3 \beta _1+\alpha _3 \beta _2 \beta _1-2 \alpha _3 \beta _3 \beta _1 &=0,&\\
 	2 \alpha _2 \beta _2 \alpha _4-\alpha _3 \beta _2 \alpha _4+2 \beta _2 \beta _4 \alpha _4-2 \beta _3 \beta _4 \alpha _4-\alpha _4^2 \beta _1 &=0.&\\
 \end{array}$

For $A_1$ the system above has the form

$\begin{array}{rr}
 \alpha _4 \beta _1^2-\alpha _1 \beta _1+\alpha _1 \alpha _2 \beta _1-2 \alpha _2 \beta _1-2 \beta _1 &=0,\\
 \alpha _1^2+\alpha _2 \alpha _1-2 \alpha _4 \beta _1 \alpha _1+\alpha _1-\alpha _2+2 \alpha _2^2 \beta _1+\alpha _2 \beta _1-\beta _1-1 &=0,\\
 \alpha _2 \alpha _1^2+\alpha _1^2-3 \alpha _2 \alpha _1+\alpha _4 \beta _1 \alpha _1-3 \alpha _1+\alpha _2+2 \alpha _2 \beta _1-\alpha _4 \beta _1+\beta _1+1 &=0,\\
 2 \alpha _4 \alpha _1^2-2 \alpha _2^2 \alpha _1-\alpha _2 \alpha _1-2 \alpha _4 \alpha _1+\alpha _1+3 \alpha _2^2+3 \alpha _2-\alpha _4 \beta _1 &=0, \\
 2 \alpha _2 \alpha _1^2-3 \alpha _1^2-3 \alpha _2 \alpha _1+2 \alpha _4 \beta _1 \alpha _1-\alpha _1-\alpha _2 \beta _1+\alpha _4 \beta _1-2 \beta _1 &=0,\\
4 \alpha _4 \alpha _1^2-4 \alpha _2^2 \alpha _1+\alpha _2 \alpha _1+\alpha _4 \alpha _1+\alpha _1+\alpha _2^2+\alpha _2-\alpha _4+\alpha _4 \beta _1&=0, \\

 2 \alpha _1 \alpha _2^2-3 \alpha _2^2-\alpha _1 \alpha _2+2 \alpha _4 \beta _1 \alpha _2-3 \alpha _2-\alpha _1 \alpha _4+\alpha _4 \beta _1-1&=0, \\
 4 \alpha _2^3+3 \alpha _2^2-4 \alpha _1 \alpha _4 \alpha _2-2 \alpha _4 \alpha _2-3 \alpha _1 \alpha _4-\alpha _4&=0, \\
 \alpha _1 \alpha _2^2-\alpha _2^2-\alpha _1 \alpha _2+\alpha _4 \beta _1 \alpha _2-\alpha _2+\alpha _1-\alpha _1 \alpha _4-1 &=0,\\

 2 \alpha _2^3-2 \alpha _1 \alpha _4 \alpha _2-\alpha _4 \alpha _2-\alpha _4 &=0,\\
 \alpha _2^2-\alpha _1 \alpha _4 \alpha _2+2 \alpha _4 \alpha _2+\alpha _4-\alpha _4^2 \beta _1&=0, \\
  \end{array}$

 $\begin{array}{rr}	
 2 \alpha _4 \alpha _2^2-2 \alpha _1 \alpha _4^2-\alpha _4^2 &=0,\\
 2 \beta _1 \alpha _1^2-3 \beta _1 \alpha _1+2 \alpha _2 \beta _1^2+2 \beta _1^2+2 \beta _1 &=0,\\
 \alpha _1^2+\alpha _2 \beta _1 \alpha _1-\beta _1 \alpha _1-3 \alpha _1+\alpha _4 \beta _1^2-\alpha _2 \beta _1+2 \beta _1+1&=0, \\
 2 \alpha _1^3-7 \alpha _1^2+2 \alpha _2 \beta _1 \alpha _1+2 \beta _1 \alpha _1+5 \alpha _1-4 \alpha _2 \beta _1-2 \beta _1-1&=0, \\
 \alpha _2 \alpha _1^2-4 \alpha _2 \alpha _1+\alpha _4 \beta _1 \alpha _1+\alpha _1+3 \alpha _2+\alpha _2 \beta _1-2 \alpha _4 \beta _1 &=0,\\
 4 \alpha _1^3-3 \alpha _1^2+4 \alpha _2 \beta _1 \alpha _1+4 \beta _1 \alpha _1+\alpha _1+3 \alpha _2 \beta _1+2 \beta _1&=0, \\
 2 \alpha _2 \alpha _1^2-\alpha _1^2-3 \alpha _2 \alpha _1+2 \alpha _4 \beta _1 \alpha _1+2 \alpha _1+3 \alpha _2+\alpha _2 \beta _1+\alpha _4 \beta _1 &=0,\\
 4 \alpha _2 \alpha _1^2-\alpha _1^2-7 \alpha _2 \alpha _1-2 \alpha _1+2 \alpha _2+4 \alpha _2^2 \beta _1+3 \alpha _2 \beta _1-\alpha _4 \beta _1+\beta _1+1 &=0,\\
 2 \alpha _1 \alpha _2^2-5 \alpha _2^2-\alpha _1 \alpha _2+2 \alpha _4 \beta _1 \alpha _2+\alpha _1 \alpha _4+\alpha _4+\alpha _4 \beta _1&=0, \\
 2 \alpha _2 \alpha _1^2+\alpha _1^2-2 \alpha _2 \alpha _1-3 \alpha _1+2 \alpha _2^2 \beta _1+3 \alpha _2 \beta _1+\alpha _4 \beta _1+\beta _1+1 &=0, \\
 \alpha _1 \alpha _2^2+\alpha _4 \beta _1 \alpha _2-2 \alpha _1 \alpha _4+\alpha _4+\alpha _4 \beta _1&=0, \\
 2 \alpha _4 \alpha _1^2-\alpha _2 \alpha _1-2 \alpha _4 \alpha _1-\alpha _1-\alpha _2^2-2 \alpha _2+2 \alpha _2 \alpha _4 \beta _1+\alpha _4 \beta _1 &=0,\\
 \beta _1 \alpha _4^2-\alpha _1 \alpha _4+\alpha _1 \alpha _2 \alpha _4-2 \alpha _2 \alpha _4 &=0. \\
\end{array}$

The last equation of the system is written  $\alpha_4(\beta _1 \alpha _4-\alpha _1 +\alpha _1 \alpha _2 -2 \alpha _2 )=0$ from that we get the following options:
\begin{itemize}
\item[Case 1] If $\alpha_4=0$ then from the equation 4 we have $\alpha_2=0$ or $\alpha_2=-\frac{3}{4}.$
\begin{itemize}
\item If $\alpha_2=0$ then from the equations 16 and 1 we get $\alpha_1=0,\ \beta_1=0$ which implies a contradiction with the equation 2.
\item If $\alpha_2=-\frac{3}{4}$ then applying the equations 22 and 1 we get $\alpha_1=0,\ \beta_1=0$ which implies a contradiction with the equation 2.
\end{itemize}
\item[Case 2] If $\alpha_4\neq 0$ then $-\beta _1 \alpha _4+\alpha _1 -\alpha _1 \alpha _2 +2 \alpha _2 =0$ and the equation 1 leads to $\beta_1=0$. Then the equation 14 implies $\alpha_1^2- 3 \alpha_1 +1=0$ (so $\alpha_1$ cannot be 0). Now using the equation 12 one has $4\alpha_1^3- 3 \alpha_1^2 +\alpha_1= \alpha_1(4\alpha_1^2- 3 \alpha_1 +1)=\alpha_1(3\alpha_1^2)=0$ hence, $\alpha_1 =0$ which is a contradiction, therefore $A_1$ is not a Malcev algebra.
\end{itemize}
Substituting the structure constants of $A_2$ into the general system we derive the system equations
\begin{equation*}\begin{array}{rr}
\alpha _1^3+\beta _2 \alpha _1^2-3 \alpha _1^2-\beta _2 \alpha _1+3 \alpha _1-\beta _1^2+\beta _2^2-1=0,&\\
 2 \alpha _1^3-5 \alpha _1^2+2 \beta _2 \alpha _1+4 \alpha _1-\beta _2-1 =0,&\\
 \alpha _1^2+3 \beta _2 \alpha _1-2 \alpha _1+2 \beta _2^2-3 \beta _2+1=0,& \\
\alpha _1^3-3 \beta _2 \alpha _1^2-2 \alpha _1^2+\beta _2 \alpha _1+\alpha _1=0, &\\
 2 \beta _1 \alpha _1^2-4 \beta _1 \alpha _1+2 \beta _1-\beta _1 \beta _2=0, &\\
 2 \alpha _1^2-\beta _2 \alpha _1+\alpha _1+\beta _2^2-1=0, &\\

 \beta _2^2-\alpha _1 \beta _2+\beta _2-\alpha _1=0, &\\
 \alpha _1^2-2 \alpha _1-\beta _2^2+1=0, &\\
 \beta _2^2+3 \alpha _1 \beta _2-2 \beta _2=0. & \\
 \alpha _1 \beta _1-\beta _2 \beta _1+\beta _1=0, &\\
 \alpha _1 \beta _1+\beta _2 \beta _1-\beta _1=0, &\\
 \alpha _1^2+\beta _2 \alpha _1-\alpha _1=0, &\\
 2 \beta _1-\alpha _1 \beta _1=0, &\\
 \alpha _1 \beta _1-\beta _1 \beta _2=0,& \\
 \beta _2 \beta _1+\beta _1=0, &\\
 2 \beta _2-1=0, &\\
 \beta _1=0,& \\

\end{array}
\end{equation*} It is easy to see the solution $\beta _1=0, \  \beta _2=\frac{1}{2}$ and $\alpha _1=\frac{1}{2}$ to the system.

For $A_3$ the general system has the form
\begin{equation*}\begin{array}{rrrr}
 2 \beta _1-\beta _1 \beta _2=0,& \ \
 \beta _2^2+2 \beta _1-1 &=0,\\
 \beta _2+1=0, &\ \
 \beta _2+3 &=0,\\
 \beta _2-1=0, &\ \
 \beta _2-3&=0, \\
 4=0, &\ \
 -1 &=0,\\
 2=0, &\ \
 2 \beta _1^2-\beta _2 \beta _1+2 \beta _1&=0, \\
 \beta _2^2+\beta _1-1=0,& \ \
 4 \beta _1+\beta _2+1&=0, \\
 \beta _2^2-2 \beta _2+3=0,&\ \
 3 \beta _1 \beta _2-3 \beta _1&=0,\\
 \beta _2^2-\beta _2-4 \beta _1-2=0, &\ \
 5-3 \beta _2&=0, \\
 \beta _2^2-2 \beta _2-2 \beta _1=0,&\ \
 \beta _2 &=0. \\
\end{array}
\end{equation*} which is a contradiction.

For $A_4$ the system has the form
\begin{equation*}
\begin{array}{rr}
 \alpha _1^3+\beta _2 \alpha _1^2-3 \alpha _1^2-\beta _2 \alpha _1+3 \alpha _1+\beta _2^2-1&=0, \\
 2 \alpha _1^3-5 \alpha _1^2+2 \beta _2 \alpha _1+4 \alpha _1-\beta _2-1&=0, \\
 \alpha _1^3-3 \beta _2 \alpha _1^2-2 \alpha _1^2+\beta _2 \alpha _1+\alpha _1 &=0.\\
\end{array}
\end{equation*} The equation 3 is written as follows
\[\alpha_1(\alpha_1^{2}-3\alpha_1\beta_2-2 \alpha_1+\beta_2+1)=0\]
\begin{itemize}
\item[Case 1:] If $\alpha _1=0$ the equation 2 implies $\beta_2=-1$ therefore, $A_4(0,-1)$ is a Malcev algebra.
\item[Case 2:] If $\alpha_1 \neq 0$ then $\alpha_1^{2}-3\alpha_1\beta_2-2 \alpha_1+\beta_2+1=0$  and the equation 2 gives
$$2 \alpha _1^3-5 \alpha _1^2+2 \alpha _1(\beta _2+1) +2 \alpha _1-(\beta _2+1)=0$$ from where we get
$2 \alpha _1^3-5 \alpha _1^2+2 \alpha _1-(\beta _2+1)(2 \alpha _1-1)=0.$ Hence, $3 \beta _2\alpha_1(-2 \alpha _1+1)=0$ and we obtain $\beta_2=0,\ \alpha_1=1$ or $\alpha _1=\frac{1}{2}, \ \beta _2=\frac{1}{2}.$
\end{itemize}
Substutiting the structure constants of algebras $A_5$, $A_6$, $A_7$ $A_9,$ $A_{10}$ and $A_{11}$ we can easily see that the systems obtained are contradictions.

As for the algebras $A_8$ the system has the solution
$\alpha _1=0.$
Finally, the set of structure constants of $ A_{12}$ satisfies the general system of equations.

 \item $I_{24}$ case. $A(A(A\otimes I)\otimes I)((u\otimes v\otimes w+v\otimes w\otimes u+w\otimes u\otimes v)\otimes u)$\\ $=-A(A\otimes A)(u\otimes v\otimes u\otimes w)-A(A(I\otimes A)\otimes I)(v\otimes u\otimes w\otimes u)-A(A(A\otimes I)\otimes I)(u\otimes w\otimes u\otimes v).$\\
Here is the system of equations in terms of structure constants\\
 $\begin{array}{rrr}
 	4 \alpha _4 \alpha _1^2+5 \alpha _2^2 \alpha _1+3 \alpha _2 \alpha _3 \alpha _1+3 \alpha _4 \beta _2 \alpha _1+\alpha _4 \beta _3 \alpha _1+2 \alpha _2 \beta _4 \alpha _1+2 \alpha _3 \beta _4 \alpha _1&&\\
 	+\alpha _4 \beta _2^2+\alpha _4 \beta _3^2+3 \alpha _2 \alpha _4 \beta _1+5 \alpha _3 \alpha _4 \beta _1+\alpha _2^2 \beta _2+\alpha _3^2 \beta _2+4 \alpha _2 \alpha _3 \beta _2&&\\+2 \alpha _2 \alpha _3 \beta _3+2 \alpha _4 \beta _2 \beta _3+2 \alpha _4 \beta _1 \beta _4+3 \alpha _3 \beta _2 \beta _4+3 \alpha _3 \beta _3 \beta _4 &=0,&\\
 	2 \alpha _4 \alpha _1^2+4 \alpha _2^2 \alpha _1+\alpha _3^2 \alpha _1+5 \alpha _2 \alpha _3 \alpha _1+3 \alpha _4 \beta _2 \alpha _1+3 \alpha _4 \beta _3 \alpha _1+2 \alpha _3 \beta _4 \alpha _1&&\\
 	+2 \alpha _4 \beta _3^2+5 \alpha _2 \alpha _4 \beta _1+3 \alpha _3 \alpha _4 \beta _1+2 \alpha _3^2 \beta _2+2 \alpha _2 \alpha _3 \beta _2+\alpha _2^2 \beta _3+\alpha _3^2 \beta _3&&\\+2 \alpha _2 \alpha _3 \beta _3+2 \alpha _4 \beta _2 \beta _3+4 \alpha _4 \beta _1 \beta _4+\alpha _3 \beta _2 \beta _4+3 \alpha _3 \beta _3 \beta _4&=0,& \\
 	2 \beta _1 \alpha _2^2+\beta _2^2 \alpha _2+\beta _3^2 \alpha _2+2 \alpha _3 \beta _1 \alpha _2+3 \alpha _1 \beta _2 \alpha _2+2 \beta _2 \beta _3 \alpha _2+3 \beta _1 \beta _4 \alpha _2+2 \alpha _3 \beta _2^2&&\\
 	+2 \beta _1 \beta _4^2+4 \alpha _1 \alpha _4 \beta _1+\alpha _1 \alpha _3 \beta _2+3 \alpha _4 \beta _1 \beta _2+5 \alpha _4 \beta _1 \beta _3+2 \alpha _3 \beta _2 \beta _3+\beta _2^2 \beta _4&&\\+4 \beta _3^2 \beta _4+3 \alpha _3 \beta _1 \beta _4+2 \alpha _1 \beta _2 \beta _4+5 \beta _2 \beta _3 \beta _4&=0,&\\
 	\beta _1 \alpha _2^2+2 \alpha _3 \beta _1 \alpha _2+3 \alpha _1 \beta _2 \alpha _2+2 \beta _2 \beta _3 \alpha _2+\beta _1 \beta _4 \alpha _2+\alpha _3 \beta _2^2+\alpha _3 \beta _3^2+4 \beta _1 \beta _4^2&&\\
 	+\alpha _3^2 \beta _1+2 \alpha _1 \alpha _4 \beta _1+3 \alpha _1 \alpha _3 \beta _2+5 \alpha _4 \beta _1 \beta _2+3 \alpha _4 \beta _1 \beta _3+4 \alpha _3 \beta _2 \beta _3&&\\+5 \beta _3^2 \beta _4+3 \alpha _3 \beta _1 \beta _4+2 \alpha _1 \beta _2 \beta _4+2 \alpha _1 \beta _3 \beta _4+3 \beta _2 \beta _3 \beta _4 &=0,&\\
 \end{array}$

   $\begin{array}{rrr}
 	\beta _1 \alpha _3^2+\beta _2^2 \alpha _3+\beta _3^2 \alpha _3+\alpha _2 \beta _1 \alpha _3+2 \alpha _1 \beta _2 \alpha _3+3 \beta _2 \beta _3 \alpha _3+2 \beta _1 \beta _4 \alpha _3+\beta _1 \beta _4^2&&\\
 	+2 \alpha _1 \alpha _2 \beta _2+\alpha _4 \beta _1 \beta _2+\alpha _4 \beta _1 \beta _3+\alpha _2 \beta _2 \beta _3+3 \beta _3^2 \beta _4+\alpha _2 \beta _1 \beta _4&&\\+\alpha _1 \beta _2 \beta _4+2 \beta _2 \beta _3 \beta _4&=0,& \\
 	
 	2 \alpha _1 \alpha _2^2+\beta _3 \alpha _2^2+3 \alpha _1 \alpha _3 \alpha _2+2 \alpha _4 \beta _1 \alpha _2+\alpha _3 \beta _2 \alpha _2+2 \alpha _3 \beta _3 \alpha _2+\alpha _1 \alpha _3^2+2 \alpha _4 \beta _3^2&&\\
 	+2 \alpha _3 \alpha _4 \beta _1+\alpha _3^2 \beta _2+\alpha _1 \alpha _4 \beta _2+\alpha _3^2 \beta _3+\alpha _1 \alpha _4 \beta _3+2 \alpha _4 \beta _2 \beta _3&&\\+\alpha _4 \beta _1 \beta _4+\alpha _3 \beta _3 \beta _4 &=0,&\\
 	2 \beta _1 \alpha _2^2+\beta _2^2 \alpha _2+\beta _3^2 \alpha _2+2 \alpha _3 \beta _1 \alpha _2+\alpha _1 \beta _2 \alpha _2+2 \beta _2 \beta _3 \alpha _2+\beta _1 \beta _4 \alpha _2+\alpha _3 \beta _2^2&&\\
 	+\alpha _1 \alpha _4 \beta _1+2 \alpha _4 \beta _1 \beta _2+2 \alpha _4 \beta _1 \beta _3+\alpha _3 \beta _2 \beta _3+\beta _2^2 \beta _4+2 \beta _3^2 \beta _4&&\\+\alpha _3 \beta _1 \beta _4+3 \beta _2 \beta _3 \beta _4&=0,& \\

 	\alpha _4 \alpha _1^2+3 \alpha _2^2 \alpha _1+2 \alpha _2 \alpha _3 \alpha _1+2 \alpha _4 \beta _2 \alpha _1+\alpha _4 \beta _3 \alpha _1+\alpha _3 \beta _4 \alpha _1+\alpha _4 \beta _2^2+\alpha _2 \alpha _4 \beta _1&&\\
 	+\alpha _3 \alpha _4 \beta _1+\alpha _2^2 \beta _2+\alpha _3^2 \beta _2+3 \alpha _2 \alpha _3 \beta _2+\alpha _2 \alpha _3 \beta _3+\alpha _4 \beta _2 \beta _3&&\\+2 \alpha _3 \beta _2 \beta _4+2 \alpha _3 \beta _3 \beta _4 &=0,&\\
 	2 \alpha _2^3+2 \alpha _3 \alpha _2^2+\beta _4 \alpha _2^2+7 \alpha _1 \alpha _4 \alpha _2+6 \alpha _4 \beta _2 \alpha _2+2 \alpha _4 \beta _3 \alpha _2+3 \alpha _3 \beta _4 \alpha _2+4 \alpha _3 \beta _4^2&&\\
 	+\alpha _1 \alpha _3 \alpha _4+2 \alpha _4^2 \beta _1+5 \alpha _3 \alpha _4 \beta _2+\alpha _3 \alpha _4 \beta _3+4 \alpha _1 \alpha _4 \beta _4&&\\+5 \alpha _4 \beta _2 \beta _4+3 \alpha _4 \beta _3 \beta _4 &=0,&\\
 	\beta _3^3+\alpha _1 \beta _3^2+\beta _2 \beta _3^2+\alpha _2 \beta _1 \beta _3+2 \alpha _3 \beta _1 \beta _3+2 \alpha _1 \beta _2 \beta _3+3 \beta _1 \beta _4 \beta _3+2 \alpha _4 \beta _1^2&&\\
 	+3 \alpha _1 \alpha _2 \beta _1+2 \alpha _1 \alpha _3 \beta _1+\alpha _1^2 \beta _2+\alpha _2 \beta _1 \beta _2+2 \alpha _3 \beta _1 \beta _2&&\\+\alpha _1 \beta _1 \beta _4+\beta _1 \beta _2 \beta _4 &=0,&\\

 	\alpha _2^3+\alpha _3 \alpha _2^2+\beta _4 \alpha _2^2+3 \alpha _1 \alpha _4 \alpha _2+2 \alpha _4 \beta _2 \alpha _2+\alpha _4 \beta _3 \alpha _2+2 \alpha _3 \beta _4 \alpha _2+\alpha _3 \beta _4^2&&\\
 	+\alpha _1 \alpha _3 \alpha _4+2 \alpha _4^2 \beta _1+2 \alpha _3 \alpha _4 \beta _2+\alpha _3 \alpha _4 \beta _3+\alpha _1 \alpha _4 \beta _4&&\\+2 \alpha _4 \beta _2 \beta _4+3 \alpha _4 \beta _3 \beta _4&=0,& \\
 	2 \beta _3^3+\alpha _1 \beta _3^2+2 \beta _2 \beta _3^2+2 \alpha _2 \beta _1 \beta _3+6 \alpha _3 \beta _1 \beta _3+3 \alpha _1 \beta _2 \beta _3+7 \beta _1 \beta _4 \beta _3+2 \alpha _4 \beta _1^2&&\\
 	+3 \alpha _1 \alpha _2 \beta _1+5 \alpha _1 \alpha _3 \beta _1+4 \alpha _1^2 \beta _2+\alpha _2 \beta _1 \beta _2+5 \alpha _3 \beta _1 \beta _2&&\\+4 \alpha _1 \beta _1 \beta _4+\beta _1 \beta _2 \beta _4 &=0,&\\
 	2 \beta _2 \alpha _2^2+5 \alpha _4 \beta _1 \alpha _2+2 \alpha _3 \beta _2 \alpha _2+3 \beta _2 \beta _4 \alpha _2+2 \beta _3 \beta _4 \alpha _2+4 \alpha _4 \beta _2^2+\alpha _4 \beta _3^2+5 \beta _2 \beta _4^2&&\\
 	+7 \beta _3 \beta _4^2+\alpha _3 \alpha _4 \beta _1+2 \alpha _1 \alpha _4 \beta _2+5 \alpha _4 \beta _2 \beta _3+6 \alpha _4 \beta _1 \beta _4+3 \alpha _3 \beta _2 \beta _4 &=0,&\\
 	7 \alpha _2 \alpha _1^2+5 \alpha _3 \alpha _1^2+6 \alpha _4 \beta _1 \alpha _1+3 \alpha _3 \beta _2 \alpha _1+2 \alpha _2 \beta _3 \alpha _1+3 \alpha _3 \beta _3 \alpha _1+2 \alpha _3 \beta _3^2+\alpha _2^2 \beta _1&&\\
 	+4 \alpha _3^2 \beta _1+5 \alpha _2 \alpha _3 \beta _1+\alpha _4 \beta _1 \beta _2+5 \alpha _4 \beta _1 \beta _3+2 \alpha _3 \beta _2 \beta _3+2 \alpha _3 \beta _1 \beta _4 &=0,&\\
 	\alpha _2^3+2 \alpha _3 \alpha _2^2+\alpha _3^2 \alpha _2+\alpha _1 \alpha _4 \alpha _2+2 \alpha _4 \beta _2 \alpha _2+4 \alpha _4 \beta _3 \alpha _2+\alpha _3 \beta _4 \alpha _2+\alpha _1 \alpha _3 \alpha _4&&\\
 	+\alpha _4^2 \beta _1+2 \alpha _3 \alpha _4 \beta _2+\alpha _3 \alpha _4 \beta _3+\alpha _3^2 \beta _4+\alpha _4 \beta _2 \beta _4+5 \alpha _4 \beta _3 \beta _4&=0,& \\
 	\beta _3^3+2 \beta _2 \beta _3^2+\beta _2^2 \beta _3+4 \alpha _2 \beta _1 \beta _3+2 \alpha _3 \beta _1 \beta _3+\alpha _1 \beta _2 \beta _3+\beta _1 \beta _4 \beta _3+\alpha _4 \beta _1^2+\alpha _1 \beta _2^2&&\\
 	+5 \alpha _1 \alpha _2 \beta _1+\alpha _1 \alpha _3 \beta _1+\alpha _2 \beta _1 \beta _2+2 \alpha _3 \beta _1 \beta _2+\beta _1 \beta _2 \beta _4 &=0,&\\
%
 	\beta _2 \alpha _2^2+\alpha _4 \beta _1 \alpha _2+\alpha _3 \beta _2 \alpha _2+2 \beta _2 \beta _4 \alpha _2+\alpha _4 \beta _2^2+\alpha _4 \beta _3^2+2 \beta _2 \beta _4^2+4 \beta _3 \beta _4^2+\alpha _3 \alpha _4 \beta _1&&\\
 	+2 \alpha _1 \alpha _4 \beta _2+2 \alpha _4 \beta _2 \beta _3+3 \alpha _4 \beta _1 \beta _4+3 \alpha _3 \beta _2 \beta _4 &=0,&\\
 	\beta _2 \alpha _2^2+2 \alpha _3 \beta _2 \alpha _2+\beta _2 \beta _4 \alpha _2+\alpha _4 \beta _2^2+\beta _2 \beta _4^2+5 \beta _3 \beta _4^2+\alpha _3 \alpha _4 \beta _1+\alpha _3^2 \beta _2+\alpha _1 \alpha _4 \beta _2&&\\
 	+4 \alpha _4 \beta _2 \beta _3+\alpha _4 \beta _1 \beta _4+3 \alpha _3 \beta _2 \beta _4+2 \alpha _3 \beta _3 \beta _4&=0,& \\
 	5 \alpha _2 \alpha _1^2+\alpha _3 \alpha _1^2+\alpha _4 \beta _1 \alpha _1+2 \alpha _2 \beta _2 \alpha _1+3 \alpha _3 \beta _2 \alpha _1+\alpha _3 \beta _3 \alpha _1+\alpha _3 \beta _2^2+\alpha _3 \beta _3^2+\alpha _3^2 \beta _1&&\\
 	+4 \alpha _2 \alpha _3 \beta _1+\alpha _4 \beta _1 \beta _2+2 \alpha _3 \beta _2 \beta _3+\alpha _3 \beta _1 \beta _4 &=0,&\\
 	4 \alpha _2 \alpha _1^2+2 \alpha _3 \alpha _1^2+3 \alpha _4 \beta _1 \alpha _1+3 \alpha _3 \beta _2 \alpha _1+2 \alpha _3 \beta _3 \alpha _1+\alpha _3 \beta _3^2+\alpha _2^2 \beta _1+\alpha _3^2 \beta _1&&\\
 	+2 \alpha _2 \alpha _3 \beta _1+\alpha _4 \beta _1 \beta _2+\alpha _4 \beta _1 \beta _3+\alpha _3 \beta _2 \beta _3+2 \alpha _3 \beta _1 \beta _4 &=0,&\\
 	4 \alpha _4 \alpha _2^2+\alpha _3 \alpha _4 \alpha _2+6 \alpha _4 \beta _4 \alpha _2+\alpha _1 \alpha _4^2+6 \alpha _4 \beta _4^2+4 \alpha _4^2 \beta _2+\alpha _4^2 \beta _3+\alpha _3 \alpha _4 \beta _4 &=0,&\\
 	6 \beta _1 \alpha _1^2+\beta _1 \beta _2 \alpha _1+6 \beta _1 \beta _3 \alpha _1+\alpha _2 \beta _1^2+4 \alpha _3 \beta _1^2+4 \beta _1 \beta _3^2+\beta _1 \beta _2 \beta _3+\beta _1^2 \beta _4&=0,& \\
 	6 \beta _4^3+10 \alpha _4 \beta _2 \beta _4+2 \alpha _4 \beta _3 \beta _4+\alpha _4^2 \beta _1+4 \alpha _2 \alpha _4 \beta _2+\alpha _3 \alpha _4 \beta _2 &=0,&\\
 	6 \alpha _1^3+2 \alpha _2 \beta _1 \alpha _1+10 \alpha _3 \beta _1 \alpha _1+\alpha _4 \beta _1^2+\alpha _3 \beta _1 \beta _2+4 \alpha _3 \beta _1 \beta _3 &=0.&\\

 \end{array}$

 \item $I_{25}$ case. $A(A\otimes I)(u\otimes v\otimes w)= A(I\otimes A )(u\otimes v\otimes w)+A(I\otimes A)(u\otimes w\otimes v).$\\
 The system of equations is
 \begin{center}
 $\begin{array}{rrrr}
 	\alpha _1 \alpha _3-\alpha _4 \beta _1+\alpha _2 \beta _2+\alpha _2 \beta _3&=0,&
 \alpha _1^2+2 \alpha _2 \beta _1-\alpha _3 \beta _1 &=0,\\
 	\alpha _1 \alpha _3-\beta _2 \alpha _3+\alpha _2 \beta _2+\alpha _2 \beta _3 &=0, &
 	\beta _2^2+\alpha _3 \beta _1 &=0,\\
 \alpha _3 \beta _2-\beta _4 \beta _2-\alpha _2 \beta _3-\alpha _3 \beta _3 &=0,&
 	\alpha _3^2+\alpha _4 \beta _2 &=0,\\
   	\alpha _2 \alpha _4-2 \alpha _3 \alpha _4-\beta _4 \alpha _4 &=0,&
 	\alpha _1 \beta _1+2 \beta _2 \beta _1-\beta _3 \beta _1 &=0,\\ 	
 	\alpha _2^2-2 \beta _4 \alpha _2-2 \alpha _1 \alpha _4+\alpha _4 \beta _2 &=0, &
 	2 \alpha _4 \beta _1-\alpha _2 \beta _2+\beta _2 \beta _4 &=0,\\	
 	\beta _3^2-2 \alpha _1 \beta _3+\alpha _3 \beta _1-2 \beta _1 \beta _4 &=0,&
 	\alpha _1 \alpha _3-\beta _3 \alpha _3+2 \alpha _4 \beta _1 &=0, \\
 	\alpha _4 \beta _1-\alpha _2 \beta _3-\alpha_3 \beta _3-\beta _2 \beta _4  &=0,&
 	\beta _4^2-\alpha _4 \beta _2+2 \alpha _4 \beta _3 &=0,\\
 \alpha _3^2+\alpha _2 \alpha _3-\beta _4 \alpha _3-\alpha _1 \alpha _4&& \beta _2^2-\alpha _1 \beta _2+\beta _3 \beta _2+\alpha _2 \beta _1&\\
 +\alpha _4 \beta _2+\alpha _4 \beta _3 &=0,& +\alpha _3 \beta _1-\beta _1 \beta _4 &=0.
 \end{array}$
\end{center}

  \item $I_{26}$ case.  $ A(A\otimes I)(u\otimes v\otimes w)= -A(I\otimes A )(u\otimes v\otimes w)-A(I\otimes A)(u\otimes w\otimes v).$\\
  The identity is equivalent to the system of equations\\
  $\begin{array}{rr}
  2 \alpha _1 \alpha _2+\beta _2 \alpha _2+\beta _3 \alpha _2+\alpha _1 \alpha _3+\alpha _4 \beta _1 =0,&
  \alpha _2^2+2 \beta _4 \alpha _2+2 \alpha _1 \alpha _4+\alpha _4 \beta _2 =0
  	,\\
  	2 \alpha _1 \alpha _2+\beta _2 \alpha _2+\beta _3 \alpha _2+\alpha _1 \alpha _3+\alpha _3 \beta _2 =0,&
  \beta _3^2+2 \alpha _1 \beta _3+\alpha _3 \beta _1+2 \beta _1 \beta _4 =0
  ,\\
  \alpha _3^2+2 \alpha _2 \alpha _3+\alpha _4 \beta _2+2 \alpha _4 \beta _3 =0, &
  	3 \alpha _1^2+2 \alpha _2 \beta _1+\alpha _3 \beta _1 =0 ,\\
  	\alpha _3^2+\alpha _2 \alpha _3+\beta _4 \alpha _3+\alpha _1 \alpha _4+\alpha _4 \beta _2+\alpha _4 \beta _3 =0, &
  	\alpha _2 \alpha _4+2 \alpha _3 \alpha _4+3 \beta _4 \alpha _4 =0, \\
  \beta _2^2+\alpha _1 \beta _2+\beta _3 \beta _2+\alpha _2 \beta _1+\alpha _3 \beta _1+\beta _1 \beta _4 =0,&
  	3 \alpha _1 \beta _1+2 \beta _2 \beta _1+\beta _3 \beta _1 =0,\\
  	\beta _2^2+2 \beta _3 \beta _2+2 \alpha _2 \beta _1+\alpha _3 \beta _1 =0,&
  	2 \alpha _4 \beta _1+\alpha _2 \beta _2+3 \beta _2 \beta _4 =0, \\
  	\alpha _3 \beta _2+\beta _4 \beta _2+\alpha _2 \beta _3+\alpha _3 \beta _3+2 \beta _3 \beta _4 =0,&
  	3 \alpha _1 \alpha _3+\beta _3 \alpha _3+2 \alpha _4 \beta _1 =0,\\
  	\alpha _4 \beta _1+\alpha _2 \beta _3+\alpha _3 \beta _3+\beta _2 \beta _4+2 \beta _3 \beta _4 =0,&
  	3 \beta _4^2+\alpha _4 \beta _2+2 \alpha _4 \beta _3  =0.
  \end{array}$

 \item $I_{27}$ case. $A((A\otimes I)-(I\otimes A))(u\otimes v\otimes w)=A((A\otimes I)-(I\otimes A))(v\otimes u\otimes w).$\\
 The identity can be written in terms of structure constants as a system of equations as follows\\
 $\begin{array}{rr}
 	\alpha _1 \alpha _2-\beta _3 \alpha _2-\alpha _1 \alpha _3+\alpha _4 \beta _1+\alpha _3 \beta _2-\alpha _3 \beta _3 =0,&
 	\alpha _2^2-\beta _4 \alpha _2-\alpha _1 \alpha _4+2 \alpha _4 \beta _2-\alpha _4 \beta _3 =0,\\
 \alpha _4 \beta _1-\alpha _2 \beta _2+\alpha _3 \beta _2-\alpha _2 \beta _3-\beta _2 \beta _4+\beta _3 \beta _4 =0,&
 	\beta _3^2-\alpha _1 \beta _3-\alpha _2 \beta _1+2 \alpha _3 \beta _1-\beta _1 \beta _4 =0.
 \end{array}$

 \item $I_{28}$ case. $A((A\otimes I)-(I\otimes A))(u\otimes v\otimes w)=-A((A\otimes I)-(I\otimes A))(v\otimes u\otimes w).$\\
 Here is the the system of equations\\
 $\begin{array}{rr}
 \alpha _4 \beta _1-\alpha _2 \beta _2-\alpha _3 \beta _2+\alpha _2 \beta _3+\beta _2 \beta _4-\beta _3 \beta _4 =0
 ,&
  \alpha _4 \beta _1- \alpha _2 \beta _2 =0,\\
 \alpha _1 \alpha _2-\beta _3 \alpha _2-\alpha _1 \alpha _3-\alpha _4 \beta _1+\alpha _3 \beta _2+\alpha _3 \beta _3 =0,&
  \alpha _4 \beta _1- \alpha _3 \beta _3=0 ,\\
  \alpha _3^2- \beta _4 \alpha _3- \alpha _1 \alpha _4+ \alpha _4 \beta _3=0, &
  \alpha _2 \alpha _4- \alpha _3 \alpha _4 =0,\\
  \beta _2^2- \alpha _1 \beta _2+ \alpha _2 \beta _1- \beta _1 \beta _4 =0,&
   \beta _1 \beta _3- \beta _1 \beta _2 =0,\\
 \beta _3^2-\alpha _1 \beta _3+\alpha _2 \beta _1-\beta _1 \beta _4 =0,&
  \alpha _3 \beta _1- \alpha _2 \beta _1 =0,\\
 \alpha _2^2-\beta _4 \alpha _2-\alpha _1 \alpha _4+\alpha _4 \beta _3 =0
 ,&
  \alpha _4 \beta _2- \alpha _4 \beta _3 =0. \\
 \end{array}$

 \item $I_{29}$ case. $A((A\otimes I)-(I\otimes A))(u\otimes v\otimes w)=A((A\otimes I)-(I\otimes A))(w\otimes v\otimes u).$\\
 The system of equations is\\
 $\begin{array}{rr}
 \alpha _2^2-\beta _4 \alpha _2+\alpha _3^2-2 \alpha _1 \alpha _4+\alpha _4 \beta _2+\alpha _4 \beta _3-\alpha _3 \beta _4 =0,&
 2 \alpha _4 \beta _1-\alpha _2 \beta _2-\alpha _3 \beta _3 =0,\\
 \beta _2^2-\alpha _1 \beta _2+\beta _3^2+\alpha _2 \beta _1+\alpha _3 \beta _1-\alpha _1 \beta _3-2 \beta _1 \beta _4 =0. &
 \end{array}$

 \item $I_{30}$ case. $A((A\otimes I)-(I\otimes A))(u\otimes v\otimes w)=-A((A\otimes I)-(I\otimes A))(w\otimes v\otimes u).$
 \end{itemize}
 One can write the identity above as a system of equations as follows
 \begin{center}
 $\begin{array}{rr}
 \alpha _2^2-\beta _4 \alpha _2-\alpha _3^2+\alpha _4 \beta _2-\alpha _4 \beta _3+\alpha _3 \beta _4 =0
 , &
 \alpha _3 \beta _3-\alpha _2 \beta _2 =0, \\
 \beta _2^2-\alpha _1 \beta _2-\beta _3^2+\alpha _2 \beta _1-\alpha _3 \beta _1+\alpha _1 \beta _3 =0,&
 \alpha _3 \beta _1- \alpha _2 \beta _1=0,\\
  \alpha _1 \alpha _2- \beta _3 \alpha _2- \alpha _1 \alpha _3+ \alpha _3 \beta _2 =0,&
 \alpha _2 \alpha _4- \alpha _3 \alpha _4 =0,\\
   \alpha _3 \beta _2-2 \beta _4 \beta _2- \alpha _2 \beta _3+ \beta _3 \beta _4  =0
 ,&
  \alpha _4 \beta _3-\alpha _4 \beta _2 =0,\\
 \beta _1 \beta _3- \beta _1 \beta _2 =0.&
 \end{array}$
 \end{center}
\end{proof}
\begin{rem}
Note that in the case of $I_{19}$ if $Char(\mathbb{F})= 5$ then $i=2$ and thus here we correct an inaccuracy admitted in Theorem 7 from \cite{B3}.
\end{rem}

The analogues of the result above for the fields of characteristic $2$ and $3$ can be easily proved following the same manner. Here we give final results without proof as the following two theorems below. In the case of the characteristic $2$ some of the identities coincide, this is also denoted by $\cong$.
 \begin{thm} Let $Char(\mathbb{F})= 2.$  Then the following classification of two-dimensional algebras over $\mathbb{F}$ with respect to the identities $I_1-I_{30}$ is valid.
 \begin{itemize}	
 	\item[$I_1\cong I_2$.] Commutativity identity $\mathbf{u}\mathbf{v}=\mathbf{v}\mathbf{u}$.\\
 	$A_{2,2}(\alpha_1,\beta_1,1+\alpha_1),A_{3,2}(\alpha_1,1+\alpha_1),$ $A_{4,2}(\alpha_1,1+\alpha_1),$ $ A_{5,2}(0)$, $A_{10,2},$ $A_{11,2},$ $A_{12,2}$.
 	
 	\item[$I_3\cong I_4$.] Associativity identity\\ 	$A_{3,2}(1,0),$ $ A_{4,2}(1,0),$ $ A_{4,2}\left(1,1\right),$ $A_{8,2}\left(1\right),$ $A_{10,2},$ $A_{12,2}$.
 	
  	\item[$I_5$.] Well defined cube identity $\mathbf{u}^2\mathbf{u}=\mathbf{u}\mathbf{u}^2.$\\
 	 $A_{1,2}\left(1,1,0,0\right),\ A_{2,2}\left(\alpha _1,\beta _1,1+\alpha_1\right)$,\ $A_{3,2}\left(\alpha _1,1-\alpha _1\right),$ $A_{4,2}\left(\alpha _1,1+\alpha_1\right),\ \mbox{where}\ \alpha_1 \neq 0,$\\
 	 $	\ A_{4,2}\left(\alpha _1,1\right),$ $A_{5,2}(0),$ $A_{8,2}\left(1\right),$ $A_{10,2},$ $A_{11,2},$ $ A_{12,2}.$
 	 	
 	\item[$I_6\cong I_7$.] Half-commutativity identity $[\mathbf{u},\mathbf{v}]\mathbf{w}=\mathbf{w}[\mathbf{u},\mathbf{v}].$\\
 	$A_{2,2}(\alpha_1,\beta_1,1-\alpha_1),$ $ A_{3,2}(\alpha _1,1+\alpha _1),$ $A_{4,2}(\alpha_1,1-\alpha_1),$ $ A_{5,2}(0),$ $ A_{10,2},$ $A_{11,2},$ $A_{12,2}$.
 	
 	\item[$I_8\cong I_9$.] Mixed associativity identity $[\mathbf{u},\mathbf{v}]\mathbf{w}=\mathbf{u}[\mathbf{v},\mathbf{w}].$\\
 		$A_{2,2}(\alpha_1,\beta_1,1+\alpha_1),$ $A_{3,2}(\alpha _1,1+\alpha _1),$ $A_{4,2}(\alpha_1,1+\alpha_1),$ $ A_{5,2}(0),$ $ A_{10,2},$ $ A_{11,2},$ $A_{12,2}$.
 	
 	 	\item[$I_{10}\cong I_{11}$.] Flexibility identity $\mathbf{u}(\mathbf{v}\mathbf{u})=(\mathbf{u}\mathbf{v})\mathbf{u}$.\\
 	$A_{2,2}(\alpha_1,\beta_1,1+\alpha_1),$ $A_{3,2}(\alpha_1,1+\alpha_1),$ $A_{4,2}(\alpha_1,1+\alpha_1),$ $A_{4,2}(\alpha_1,1)$, where $\alpha_1\neq 0,$ $A_{5,2}(0)$,\\ $A_{8,2}(1),$ $A_{10,2},$ $A_{11,2},$ $A_{12,2}$.
 	
  	\item[$I_{12}\cong I_{13}$.] Mixed flexibility identity $\mathbf{u}[\mathbf{v},\mathbf{u}]=[\mathbf{u},\mathbf{v}]\mathbf{u}$.\\
 	$A_{2,2}(\alpha_1,\beta_1,1-\alpha_1)$, $A_{3,2}(\alpha_1,1+\alpha_1)$, $A_{4,2}(\alpha_1,1+\alpha_1)$, $ A_{5,2}(0)$, $A_{10,2}$, $A_{11,2}$, $A_{12,2}$.
 	
 	 	\item[$I_{14}\cong I_{15}$.] Left Leibniz identity $\mathbf{u}(\mathbf{v}\mathbf{w})=(\mathbf{u}\mathbf{v})\mathbf{w}+\mathbf{v}(\mathbf{u}\mathbf{w})$.\\
 	$A_{4,2}(0,1),$ $ A_{8,2}(0),$ $ A_{12,2}$.
 	
 	\item[$I_{16}\cong I_{17}$.] Mixed left Leibniz identity $\mathbf{u}[\mathbf{v},\mathbf{w}]=[\mathbf{u},\mathbf{v}]\mathbf{w}+\mathbf{v}[\mathbf{u},\mathbf{\mathbf{w}}]$.\\
 	$A_{2,2}(\alpha_1,\beta_1,1+\alpha_1),$ $A_{3,2}(\alpha_1,1+\alpha_1),$ $ A_{4,2}(\alpha_1,1+\alpha_1),$ $ A_{5,2}(0),$ $ A_{10,2},$ $A_{11,2},$ $A_{12,2}$.
 	
 	\item[$I_{18}$.] Left Poisson identity $(\mathbf{u}\mathbf{v})\mathbf{w}+(\mathbf{v}\mathbf{w})\mathbf{u}+(\mathbf{w}\mathbf{u})\mathbf{v}=0$.\\
 	$ A_{4,2}(0,1), $ $A_{5,2}(0),$ $ A_{8,2}(0),$ $ A_{12,2}$.
 	
 	\item[$I_{19}\cong I_{20}$.] Left Jordan identity $(\mathbf{u}\mathbf{v})\mathbf{u}^2=\mathbf{u}(\mathbf{v}\mathbf{u}^2)$.\\  $A_{3,2}(0,0),\ A_{3,2}(1,0),\ A_{4,2}(\alpha_1,1),\ A_{4,2}\left(\alpha_1,\sqrt{\alpha _1+\alpha _1^2}\right),\ \mbox{where}\ \alpha _1^2+\alpha_1+1\neq 0,\
 	A_{5,2}(\alpha_1),$ where $\alpha _1^2+\alpha_1+1=0,$ \ \
$A_{8,2}(1),\ A_{10,2},\ A_{12,2}.$
 	
  	\item[$I_{21}\cong I_{22}$.] Mixed left Jordan identity $[\mathbf{u},\mathbf{v}]\mathbf{u}^2=\mathbf{u}[\mathbf{v},\mathbf{u}^2].$\\
 	$A_{2,2}(\alpha_1,\beta_1,1+\alpha_1),$ $A_{3,2}(\alpha_1,1+\alpha_1),$ $ A_{4,2}(\alpha_1,1+\alpha_1),$ $A_{4,2}(0,0),$ $ A_{5,2}(0),$ $ A_{10,2},$ $A_{11,2},$ $A_{12,2}$.
 	
 \item[$I_{23}\cong I_{24}$.] Left Malcev identity $((\mathbf{u}\mathbf{v})\mathbf{w}+(\mathbf{v}\mathbf{w})\mathbf{u}+(\mathbf{w}\mathbf{u})\mathbf{v})\mathbf{u}=(\mathbf{u}\mathbf{v})(\mathbf{u}\mathbf{w})+(\mathbf{v}(\mathbf{u}\mathbf{w}))\mathbf{u} +((\mathbf{u}\mathbf{w})\mathbf{u})\mathbf{v}$.\\
 	$A_{3,2}(1,0),$ $ A_{4,2}(1,0),$ $ A_{4,2}(0,1), $ $ A_{8,2}(0), $ $ A_{10,2},$ $A_{12,2}$.
 	
 	\item[$I_{25}\cong I_{26}$.] Left Zinbiel identity $ (\mathbf{u}\mathbf{v})\mathbf{w}=\mathbf{u}(\mathbf{v}\mathbf{w}+\mathbf{w}\mathbf{v})$.\ \
 	$ A_{12,2}$.
 	
  	\item[$I_{27}\cong I_{28}$.] Left symmetric identity $[\mathbf{u},\mathbf{v},\mathbf{w}]=[\mathbf{v},\mathbf{u},\mathbf{w}]$.\\
 	$A_{3,2}(1,0),$ $A_{4,2}(1,\beta_2),$ $ A_{5,2}(1),$ $ A_{6,2}(\alpha_1,0),$ $A_{7,2}(1),$ $A_{8,2}(\alpha_1),$ $ A_{9,2},$ $A_{10,2},$ $A_{12,2}$.
 	
 \item[$I_{29}\cong I_{30}$.] Centro-symmetric identity $[\mathbf{u},\mathbf{v},\mathbf{w}]=[\mathbf{w},\mathbf{v},\mathbf{u}]$. \\  $A_{1,2}(\alpha _1, 1+\alpha _1,\alpha _1,1+\alpha _1),$ $A_{2,2}(\alpha _1,\beta_1,1+\alpha_1),$ $A_{3,2}(\alpha_1,1+\alpha_1),$ $A_{4,2}(\alpha_1,1+\alpha_1),$ $A_{4,2}(\alpha_1,1),\\ \mbox{where}\ \alpha_1\neq 0,$ $A_{5,2}(\alpha_1),$ $ A_{8,2}(1),$ $A_{9,2},$ $ A_{10,2},$ $ A_{11,2},$ $A_{12,2}.$
 \end{itemize}	
 	 \end{thm}
\begin{thm} In the case of $Char(\mathbb{F})= 3$ the following classification result of two-dimensional algebras over $\mathbb{F}$ with respect to the identities $I_1-I_{30}$ holds true.
	\begin{itemize}
	\item[$I_1$.] Commutativity identity $\mathbf{u}\mathbf{v}=\mathbf{v}\mathbf{u}$.\\
	$A_{2,3}(\alpha_1,\beta_1,1-\alpha_1)\cong A_{2,3}(\alpha_1,-\beta_1,1-\alpha_1),A_{3,3}(\beta_1,1),A_{4,3}(\alpha_1,,1-\alpha_1), A_{5,3}(\alpha_1),A_{9,3},A_{10,3},$ $A_{11,3},A_{12,3}$.
	
	\item[$I_2$.] Anti-commutativity identity $\mathbf{u}\mathbf{v}=-\mathbf{v}\mathbf{u}$.
	$A_{4,3}(0,-1)$.
	
	\item[$I_3$.] Associativity identity\\ $A_{2,3}\left(-1,0,-1\right),\
	 A_{4,3}\left(-1,-1\right),
	\ A_{4,3}(1,0),\ A_{4,3}\left(1,1\right),\ A_{4,3}\left(-1,0\right),\
	A_{12,3}$.	
	
	\item[$I_4$.] Anti-associativity identity $(\mathbf{u}\mathbf{v})\mathbf{w}=-\mathbf{u}(\mathbf{v}\mathbf{w}),$ 
$A_{12,3}$.
	
	\item[$I_5$.] Well defined cube identity\\
	 $A_{2,3}\left(\alpha _1,\beta _1,1-\alpha _1\right)\cong A_{2,3}\left(\alpha _1,-\beta _1,1-\alpha _1\right),
	\ A_{3,3}\left(\beta _1,1\right), A_{3,3}\left(0,-1\right),\ A_{4,3}\left(\alpha _1,1-\alpha _1\right),$\\ $A_{4,3}\left(\alpha _1,2\alpha _1-1\right),$ $A_{9,3},$ $A_{10,3},\ A_{11,3},\ A_{12,3}$.
	
	\item[$I_6$.] Half-commutativity identity $[\mathbf{u},\mathbf{v}]\mathbf{w}=\mathbf{w}[\mathbf{u},\mathbf{v}].$\\
	$A_{2,3}(\alpha_1,\beta_1,1-\alpha_1)\cong A_{2,3}(\alpha_1,-\beta_1,1-\alpha_1),A_{3,3}(\beta_1,1)$, $A_{4,3}(\alpha_1,1-\alpha_1), A_{9,3},A_{10,3},A_{11,3},A_{12,3}$.
	
	\item[$I_7$.] Anti-half-commutativity identity $[\mathbf{u},\mathbf{v}]\mathbf{w}=-\mathbf{w}[\mathbf{u},\mathbf{v}].$\\
	$A_{2,3}(\alpha_1,\beta_1,1-\alpha_1)\cong A_{2,3}(\alpha_1,-\beta_1,1-\alpha_1),A_{3,3}(\beta_1,1),A_{4,3}(\alpha_1,1-\alpha_1), A_{4,3}(\alpha_1,\alpha_1-1), A_{5,3}(0),\\ A_{8,3}(-1), A_{9,3},A_{10,3},A_{11,3},A_{12,3}$.
	
	\item[$I_8$.] Mixed associativity identity $[\mathbf{u},\mathbf{v}]\mathbf{w}=\mathbf{u}[\mathbf{v},\mathbf{w}].$\\
	$A_{2,3}(\alpha_1,\beta_1,1-\alpha_1)\cong A_{2,3}(\alpha_1,-\beta_1,1-\alpha_1), A_{3,3}(\beta_1,1), A_{4,3}(\alpha_1,1-\alpha_1), A_{9,3},A_{10,3},A_{11,3},A_{12,3}$.
	
	\item[$I_9$.] Anti-mixed-associativity identity $[\mathbf{u},\mathbf{v}]\mathbf{w}=-\mathbf{u}[\mathbf{v},\mathbf{w}].$\\
	$A_{2,3}(\alpha_1,\beta_1,1-\alpha_1)\cong A_{2,3}(\alpha_1,-\beta_1,1-\alpha_1), A_{3,3}(\beta_1,1), A_{4,3}(\alpha_1,1-\alpha_1), A_{9,3}, A_{10,3}, A_{11,3}, A_{12,3}$.
	
	\item[$I_{10}$.] Flexibility identity $\mathbf{u}(\mathbf{v}\mathbf{u})=(\mathbf{u}\mathbf{v})\mathbf{u}$.\\
	$A_{2,3}(\alpha_1,\beta_1,1-\alpha_1)\cong A_{2,3}(\alpha_1,-\beta_1,1-\alpha_1)$, $A_{3,3}(\beta_1,1),$  $A_{4,3}(\alpha_1,1-\alpha_1)$, $A_{4,3}(\alpha_1,2\alpha_1-1)$, $A_{9,3}$, $A_{10,3}$, $A_{11,3}$, $A_{12,3}$.
	
	\item[$I_{11}$.] Anti-flexibility identity $\mathbf{u}(\mathbf{v}\mathbf{u})=-(\mathbf{u}\mathbf{v})\mathbf{u}$.\ \
	$A_{12,3}.$
	
	\item[$I_{12}$.] Mixed flexibility identity $\mathbf{u}[\mathbf{v},\mathbf{u}]=[\mathbf{u},\mathbf{v}]\mathbf{u}$.\\
	$A_{2,3}(\alpha_1,\beta_1,1-\alpha_1)\cong A_{2,3}(\alpha_1,-\beta_1,1-\alpha_1),A_{3,3}(\beta_1,1),A_{4,3}(\alpha_1,1-\alpha_1), A_{4,3}(\alpha_1,\alpha_1-1), A_{5,3}(0),\\ A_{8,3}(-1), A_{9,3},A_{10,3},A_{11,3},A_{12,3}$.
	
	\item[$I_{13}$.] Mixed anti-flexibility identity $\mathbf{u}[\mathbf{v},\mathbf{u}]=-[\mathbf{u},\mathbf{v}]\mathbf{u}$.\\
	$A_{2,3}(\alpha_1,\beta_1,1-\alpha_1)\cong A_{2,3}(\alpha_1,-\beta_1,1-\alpha_1),A_{3,3}(\beta_1,1),A_{4,3}(\alpha_1,1-\alpha_1),  A_{9,3},A_{10,3},A_{11,3},A_{12,3}$.
	
	\item[$I_{14}$.] Left Leibniz identity $\mathbf{u}(\mathbf{v}\mathbf{w})=(\mathbf{u}\mathbf{v})\mathbf{w}+\mathbf{v}(\mathbf{u}\mathbf{w})$.\ \
	$A_{4,3}(0,-1), A_{8,3}(0), A_{12,3}$.
	
	\item[$I_{15}$.] Left anti-Leibniz identity $\mathbf{u}(\mathbf{v}\mathbf{w})=-(\mathbf{u}\mathbf{v})\mathbf{w}-\mathbf{v}(\mathbf{u}\mathbf{w})$.\\ $A_{2,3}(2,0,-1),A_{4,3}(1,0),A_{4,3}(1,1), A_{4,3}(-1,-1), A_{12,3}$.
	
	\item[$I_{16}$.] Mixed left Leibniz identity $\mathbf{u}[\mathbf{v},\mathbf{w}]=[\mathbf{u},\mathbf{v}]\mathbf{w}+\mathbf{v}[\mathbf{u},\mathbf{w}]$.\\
$A_{2,3}(\alpha_1,\beta_1,1-\alpha_1)\cong A_{2,3}(\alpha_1,-\beta_1,1-\alpha_1),A_{3,3}(\beta_1,1),A_{4,3}(\alpha_1,1-\alpha_1), A_{4,3}(\alpha_1,\alpha_1-1), A_{5,3}(0),\\ A_{8,3}(-1), A_{9,3},A_{10,3},A_{11,3},A_{12,3}$.

	\item[$I_{17}$.] Mixed anti-left Leibniz identity $\mathbf{u}[\mathbf{v},\mathbf{w}]=-[\mathbf{u},\mathbf{v}]\mathbf{w}-\mathbf{v}[\mathbf{u},\mathbf{w}]$.\\
	$A_{2,3}(\alpha_1,\beta_1,1-\alpha_1)\cong A_{2,3}(\alpha_1,-\beta_1,1-\alpha_1),A_{3,3}(\beta_1,1),A_{4,3}(\alpha_1,1-\alpha_1),  A_{9,3},A_{10,3},A_{11,3},A_{12,3}$.
	
	\item[$I_{18}$.] Left Poisson identity $(\mathbf{u}\mathbf{v})\mathbf{w}+(\mathbf{v}\mathbf{w})\mathbf{u}+(\mathbf{w}\mathbf{u})\mathbf{v}=0$.\\
	$A_{2,3}(0,0,-1), A_{2,3}(2,0,-1),A_{4,3}(\alpha_1,-(1-\alpha_1)^2), A_{5,3}(0), A_{8,3}(0),A_{12,3}$.

	\item[$I_{19}$.] Left Jordan identity $(\mathbf{u}\mathbf{v})\mathbf{u}^2=\mathbf{u}(\mathbf{v}\mathbf{u}^2)$.\\
	$A_{2,3}(-1,0,1),\ A_{2,3}(-1,0,-1),\ A_{4,3}(\alpha_1,-1-\alpha_1),\mbox{where}\ \alpha _1\neq -1\pm i,\
	A_{4,3}\left(\alpha_1,\sqrt{\alpha
		_1-\alpha _1^2}\right),$ $A_{4,3}\left(\alpha_1,-\sqrt{\alpha _1-\alpha _1^2}\right),\mbox{ where}\ \alpha _1\neq 0,1,$ $A_{5,3}( - 1+i),$ $A_{5,3}(- 1-i),$ $A_{8,3}(-1+i),$ $A_{8,3}(-1-i),$ $A_{10,3},$ $A_{12,3}.$
	
	\item[$I_{20}$.] Left anti-Jordan identity $(\mathbf{u}\mathbf{v})\mathbf{u}^2=-\mathbf{u}(\mathbf{v}\mathbf{u}^2)$.\ \
	$A_{4,3}(0,-1), A_{4,3}(0,0), A_{12,3}$.
		
	\item[$I_{21}$.] Mixed left Jordan identity $[\mathbf{u},\mathbf{v}]\mathbf{u}^2=\mathbf{u}[\mathbf{v},\mathbf{u}^2].$\\
	$A_{2,3}(\alpha_1,\beta_1,1-\alpha_1)\cong A_{2,3}(\alpha_1,-\beta_1,1-\alpha_1), A_{3,3}(\beta_1,1), A_{4,3}(0,-1), A_{4,3}(0,0)$, $A_{4,3}(\alpha_1,1-\alpha_1)$, $A_{9,3}$, $A_{10,3},$ $A_{11,3},$ $A_{12,3}$.
	
	\item[$I_{22}$.] Mixed anti-left Jordan identity $[\mathbf{u},\mathbf{v}]\mathbf{u}^2=-\mathbf{u}[\mathbf{v},\mathbf{u}^2].$\\
		$A_{2,3}(\alpha_1,\beta_1,1-\alpha_1)\cong A_{2,3}(\alpha_1,-\beta_1,1-\alpha_1), A_{3,3}(\beta_1,1), A_{4,3}(0,-1), A_{4,3}(0,0)$,  $A_{4,3}(\alpha_1,1-\alpha_1)$, $A_{9,3}$, $A_{10,3},$ $A_{11,3}$, $A_{12,3}$.
	
	\item[$I_{23}$.] Left Malcev identity $((\mathbf{u}\mathbf{v})\mathbf{w}+(\mathbf{v}\mathbf{w})\mathbf{u}+(\mathbf{w}\mathbf{u})\mathbf{v})\mathbf{u}=(\mathbf{u}\mathbf{v})(\mathbf{u}\mathbf{w})+(\mathbf{v}(\mathbf{u}\mathbf{w}))\mathbf{u} +((\mathbf{u}\mathbf{w})\mathbf{u})\mathbf{v}$.\\
	$A_{2,3}(-1,0,-1), A_{2,3}(0,0,-1),A_{4,3}(\alpha_1,-(1-\alpha_1)^2), A_{5,3}(0), A_{8,3}(0),A_{12,3}$.
	
	\item[$I_{24}$.] Left anti-Malcev identity $((\mathbf{u}\mathbf{v})\mathbf{w}+(\mathbf{v}\mathbf{w})\mathbf{u}+(\mathbf{w}\mathbf{u})\mathbf{v})\mathbf{u}=-(\mathbf{u}\mathbf{v})(\mathbf{u}\mathbf{w})-(\mathbf{v}(\mathbf{u}\mathbf{w}))\mathbf{u} -((\mathbf{u}\mathbf{w})\mathbf{u})\mathbf{v}$.\\ $A_{2,3}(-1,0,-1), A_{2,3}(0,0,-1),A_{4,3}(\alpha_1,-(1-\alpha_1)^2), A_{5,3}(0), A_{8,3}(0),A_{12,3}$.
		
	\item[$I_{25}$.] Left Zinbiel identity $ (\mathbf{u}\mathbf{v})\mathbf{w}=\mathbf{u}(\mathbf{v}\mathbf{w}+\mathbf{w}\mathbf{v})$.\ \
	$ A_{12,3}$.
	
	\item[$I_{26}$.] Left anti-Zinbiel identity $ (\mathbf{u}\mathbf{v})\mathbf{w}=-\mathbf{u}(\mathbf{v}\mathbf{w}+\mathbf{w}\mathbf{v})$.\\
	$A_{2,3}(-1,0,-1), A_{4,3}(-1,0),A_{4,3}(-1,-1), A_{4,3}(1,0),A_{12,3}$.

	\item[$I_{27}$.] Left symmetric identity $[\mathbf{u},\mathbf{v},\mathbf{w}]=[\mathbf{v},\mathbf{u},\mathbf{w}]$.\\
	$A_{2,3}(-1,0,-1), A_{2.3}(1,0,-1),  A_{4,3}(1,\beta_2), A_{4,3}(-1,\beta_2),A_{5,3}(-1),A_{5,3}(1), A_{8,3}(0), A_{12,3}$.
	
	\item[$I_{28}$.] Left anti-symmetric identity $[\mathbf{u},\mathbf{v},\mathbf{w}]=-[\mathbf{v},\mathbf{u},\mathbf{w}]$.\\    $A_{2,3}(-1,0,-1), A_{4,3}(-1,-1),  A_{4,3}(-1,0), A_{4,3}(1,0),A_{4,3}(1,1), A_{12,3}$.
	
	\item[$I_{29}$.] Centro-symmetric identity $[\mathbf{u},\mathbf{v},\mathbf{w}]=[\mathbf{w},\mathbf{v},\mathbf{u}]$. \\  $A_{1,3}(\alpha_1,\alpha_2, -\alpha_1-2\alpha_2,2\alpha_1+\alpha_2),$ where $\alpha_2=  \left( \pm \sqrt{2 \alpha_1}- \alpha_1+1\right)$,\\ $A_{2,3}(-1,0,-1)$, $A_{4,3}(\alpha_1,-\alpha_1+\sqrt{-\alpha_1^2-1})$,$A_{4,3}(\alpha_1,-\alpha_1-\sqrt{-\alpha_1^2-1})$,  $A_{5,3}(-1)$, $A_5(1)$, $A_{8,3}(i)$, $A_{8,3}(-i)$, $A_{12,3}$.
	
	\item[$I_{30}$.] Centro-anti-symmetric identity $[\mathbf{u},\mathbf{v},\mathbf{w}]=-[\mathbf{w},\mathbf{v},\mathbf{u}]$.\\
	$A_{2,3}(\alpha_1,\beta_1,1-\alpha_1)\cong A_{2,3}(\alpha_1,-\beta_1,1-\alpha_1)$, $A_{3,3}(\beta_1,1)$, $A_{4,3}(\alpha_1,1-\alpha_1)$, $A_{4,3}(\alpha_1,2\alpha_1-1)$, $A_{9,3}$, $A_{10,3}$, $A_{11,3}$, $A_{12,3}$.
\end{itemize}
\end{thm}


\begin{thebibliography}{99}
	\bibitem{B1} H. Ahmed, U. Bekbaev, I. Rakhimov, Classification of $2$-dimensional evolution algebras, their groups of automorphisms and derivation algebras, \textit{arXiv:} 17.  (2017).	
	\bibitem{B2} H. Ahmed, U. Bekbaev, I. Rakhimov, Complete classification of two-dimensional algebras, 12 pages, \emph{AIP Conf. Proc}, 1830, 070016, 2017, doi 10.1063/1.4980965.
	\bibitem{B3} H. Ahmed, U. Bekbaev, I. Rakhimov, Classification of two-dimensional Jordan algebras, \emph{AIP Conf. Proc}, 1905, 030003-1–030003-8; https://doi.org/10.1063/1.5012149
			\bibitem{B4}  H. Ahmed, U. Bekbaev, I. Rakhimov, On Two-Dimensional Power Associative Algebras Over Algebraically Closed Fields and $\mathbb{R}$, \emph{Lobachevskii Journal of Mathematics}, 2019, 40(1), 1--13.
			\bibitem{B5}  H. Ahmed, U. Bekbaev, I. Rakhimov, Classification of two-dimensional Jordan algebras over $\mathbb{R}$, \emph{Malaysian Journal of Mathematical Sciences}, 2018, 12(3), 287--303
		\bibitem{D2003} R. Dur\'{a}n D\'{i}az, J.M. Masq\'{e}, A.P. Dom\'{i}nguez, Classifying quadratic maps from plane to plane, \emph{Linear Algebra and Appl.}, 2003, 364, 1--12.
	\bibitem{G2011} M. Goze and E. Remm, 2-dimensional algebras, \emph{African Journal of Mathematical Physics}, 10, 2011, 81--91.
	\bibitem{P2000}  H.P. Petersson, The classification of two-dimensional
	nonassicative algebras, \emph{Result. Math}., 3, 2000, 120--154.
	\bibitem{Casado2} Y.C. Casado, \textit{Evolution algebras}, PhD thesis, Universidad de Málaga (2016) http://orcid.org/0000-0003-4299-4392.	
	\bibitem{C} J.M. Casas, M. Ladra, B.A. Omirov, U.A. Rozikov, On evolution Algebras, \textit{Algebra Colloq.}, 2014, 21, 331--342.
	\bibitem{IK} I. Kaygorodov, Yu. Volkov, The variety of 2-dimensional algebras over an algebraically closed field, \emph{Canadian Journal of Mathematics}, 2019, 71(4), 819--842
	\bibitem{G} A. Giambruno, S. Mishcenko, M. Zaicev, Codimension growth of two-dimensional non-associative algebras, \emph{Proceedings of the American Mathematical Society}, 2007, 135(11) 3405--3415.
		
	
	
	
\end{thebibliography}
\end{document}